\documentclass[11pt]{amsart}
\usepackage[T1]{fontenc}
\usepackage{color}
\usepackage{url}
\usepackage{hyperref}
\usepackage{mathtools}
\mathtoolsset{showonlyrefs}
\usepackage{anysize}
\usepackage{geometry}
\usepackage{amssymb,amsmath}
\usepackage[all,cmtip]{xy}
\usepackage{graphicx}
\marginsize{2cm}{2cm}{2cm}{2cm}
\theoremstyle{plain}
\numberwithin{equation}{section}
\newtheorem{theorem}{Theorem}[section]
\newtheorem{lemma}[theorem]{Lemma}
\newtheorem{corollary}[theorem]{Corollary}
\newtheorem{proposition}[theorem]{Proposition}
\title[Action of $\mathbb{R}$-Fuchsian groups on $\mathbb{P}_\mathbb{C}^n$]{Action of $\mathbb{R}$-Fuchsian groups on $\mathbb{P}_\mathbb{C}^n$}
\author{W. Barrera}
\address{Facultad de Matem\'aticas, Universidad Aut\'onoma de Yucat\'an, Anillo Perif\'erico Norte Tablaje Cat, M\'erida, 13615, Yucat\'an, M\'exico.}
\email{bvargas@correo.uady.mx}
\author{E. Montiel}
\address{Instituto de Matem\'aticas, UNAM, Av. Universidad, Cuernavaca, 62210, Morelos, M\'exico.}
\email{eduardo.montiel@im.unam.mx}
\author{J.P. Navarrete}
\address{Facultad de Matem\'aticas, Universidad Aut\'onoma de Yucat\'an, Anillo Perif\'erico Norte Tablaje Cat, M\'erida, 13615, Yucat\'an, M\'exico.}
\email{jp.navarrete@correo.uady.mx}

\begin{document}

\begin{abstract}
We consider discrete subgroups of the group of  orientation preserving isometries of the $m$-dimensional hyperbolic space, whose limit set is a $(m-1)$-dimensional real sphere, acting on the $n$-dimensional complex projective space for $n\geq m$, via an embedding from the group of orientation preserving isometries of the $m$-dimensional hyperbolic space to the group of holomorphic isometries of the $n$-dimensional complex hyperbolic space. We describe the Kulkarni limit set of any of these subgroups under the embedding as a real semi-algebraic set. Also, we show that the Kulkarni region of discontinuity can only have one or three connected components. We use the Sylvester's law of inertia when $n=m$. In the other cases, we use some suitable projections of the the $n$-dimensional complex projective space to the $m$-dimensional complex projective space.   
\end{abstract}

\maketitle

\section{Introduction}\label{sec1}
 Let $\Gamma$ be a discrete subgroup of SO$_+(m,1)$, the group of isometries that preserves orientation of the hyperbolic space of dimension $m$, we know that the limit set of $\Gamma$ is contained in $\mathbb{S}^{m-1}$, see \cite{cano2013complex,kapovich2008kleinian}. Here, we are taking the limit set of a discrete subgroup $\Gamma \subset \mathrm{SO}_+(m,1)$ just as the accumulation set in the sphere $\mathbb{S}^{m-1}$ of an orbit $\Gamma \cdot \mathbf{z}$, where $\mathbf{z} \in \mathbb{H}_{\mathbb{R}}^{m}$, see \cite{cano2013complex,kapovich2008kleinian}. We are interested in discrete subgroups $\Gamma \subset \mathrm{SO}_+(m,1)$ whose limit set is the whole sphere $\mathbb{S}^{m-1}$. In the remainder of this article all the discrete subgroups we will consider are of this kind, examples of such groups are given by lattices. Given $\Gamma$ a discrete subgroup of SO$_+(m,1)$ we can obtain a complex Kleinian group $\iota(\Gamma)$ whose limit set in the sense Chen and Greenberg, see \cite{chen1974hyperbolic}, is a $(m-1)$-dimensional real sphere, by the embedding 
\begin{gather} \label{eq:0}
 \iota: \mathrm{SO}_+(m,1)  \rightarrow \mathrm{SU}(n,1)\\
  G= \left(\begin{matrix} \notag
 A  & \mathbf{b}\\
\mathbf{c}  & d
\end{matrix}\right)\mapsto \left(\begin{matrix}
 A & 0 & \mathbf{b}\\
0 & I & 0 \\
\mathbf{c} & 0 & d
\end{matrix}\right)
\end{gather} 
where $A$ is the left-upper submatrix of size $m \times m$ of $G$, $\mathbf{c}$ is a row matrix of size $m$, $\mathbf{b}$ is a column vector of size $m$, $d$ is a $1 \times 1$ entry and $I$ is the identity matrix of size $(n-m) \times (n-m)$. By the work of Cano et al. in \cite{cano2017limit} we can compute the Kulkarni limit set (see \cite{kulkarni1978groups}) of $\iota(\Gamma)$ as the union of all complex projective hyperplanes in $\mathbb{P}_\mathbb{C}^n$ which are tangent to the sphere $\partial \mathbb{H}_\mathbb{C}^n$ at points in the Chen-Greenberg limit set, in other words, if we denote by $\Lambda_{Kul}(\iota(\Gamma))$ the Kulkarni limit set of $\iota(\Gamma)$, by $\mathcal{S}$ the $(m-1)$-dimensional real sphere contained in $\partial \mathbb{H}_\mathbb{R}^n$ whose points have the form $[w_1: \cdots :w_m:0:\cdots:0:w_{n+1}]$  and by $H_z$ the complex projective hyperplane tangent to $\partial \mathbb{H}_\mathbb{C}^n$ at the point $z \in \mathcal{S}$ then 
\begin{equation}\notag
\Lambda_{Kul}(\iota(\Gamma))=\underset{z\in\mathcal{S}}{\bigcup}H_{z}.
\end{equation}
We show that in this case $\Lambda_{Kul}(\iota(\Gamma))$ is a real semi-algebraic set and that the complement $\Omega_{Kul}:=\mathbb{P}_\mathbb{C}^n\setminus \Lambda_{Kul}(\iota(\Gamma))$ has three connected components when $m=2$, and is connected in other case, the precise statements are given in Theorems \ref{teo:uno}, \ref{teo:dos}, \ref{teo:tres}. The case $m=n=2$ was studied in \cite{cano2016action} using as main tool the Hermitian cross-product $\boxtimes$ defined in $\mathbb{C}^{2,1}$ by Goldman, the Hermitian cross-product of two vectors $\mathbf{z}$ and $\mathbf{w}$ is a orthogonal vector to both $\mathbf{z}$ and $\mathbf{w}$. The function $f(\mathbf{z}):=\langle i\mathbf{z} \boxtimes \overline{\mathbf{z}}, i\mathbf{z} \boxtimes \overline{\mathbf{z}} \rangle$ allows to characterize $\Lambda$ as a semi-algebraic subset of $\mathbb{P}_\mathbb{C}^2$, moreover the function $f$ gives a partition of $\mathbb{C}^{2,1} \setminus \{\mathbf{0}\}$ that under projection provides a partition of $\mathbb{P}_\mathbb{C}^2$ into $\mathrm{SO}_+(2,1)$-invariants sets. This partition allows to describe the set $\Lambda$ and the connected components of $\Omega$. When we consider  $\Gamma \subset \mathrm{SO}_+(m,1)$ acting on $\mathbb{P}_\mathbb{C}^m$, we find that the Sylvester's law of Inertia is a good replacement for the Hermitian cross-product because if we define the function 
\begin{equation}\notag 
f:= -\bigg\lvert \begin{array}{cc}
\langle\mathbf{z_1},\mathbf{z_1}\rangle & \langle\mathbf{z_1},\mathbf{z_2}\rangle\\
\langle\mathbf{z_2},\mathbf{z_1}\rangle & \langle\mathbf{z_2},\mathbf{z_2}\rangle
\end{array}\bigg\vert        
 \end{equation}
then the Sylvester's law of Inertia allows to describe the set $\Lambda$ as a real semi-algebraic subset of $\mathbb{P}_\mathbb{C}^m$. The function $f$ defined in terms of the determinant of $\left(\begin{array}{cc}
\langle\mathbf{z_1},\mathbf{z_1}\rangle & \langle\mathbf{z_1},\mathbf{z_2}\rangle\\
\langle\mathbf{z_2},\mathbf{z_1}\rangle & \langle\mathbf{z_2},\mathbf{z_2}\rangle
\end{array}\right)$ coincides with the Hermitian product $\langle i\mathbf{z} \boxtimes \overline{\mathbf{z}}, i\mathbf{z} \boxtimes \overline{\mathbf{z}} \rangle$ when $m=2$. For our purposes we just need to know if the orthogonal complement of the plane spanned by $\mathbf{z}$ and $\overline{\mathbf{z}}$, under the assumption that $\mathbf{z}$ and $\overline{\mathbf{z}}$ are linearly independent, is elliptic (positive definite), parabolic (degenerate) or hyperbolic (non-degenerate and indefinite), this is the main reason to use the Sylvester's law of Inertia. This information as well as the linearly dependence of $\mathbf{z}$ and $\overline{\mathbf{z}}$  is codified by the function $f$. In the same manner that before the function $f$ determines a partition of $\mathbb{C}^{m,1} \setminus \{\mathbf{0}\}$ which projects to a partition of $\mathbb{P}_\mathbb{C}^m$ into $\mathrm{SO}_+(m,1)$-invariants sets which are useful to describe $\Lambda$ and $\Omega$. For $2<m=n$ we have that the set $\Omega$ is connected, in contrast to the case $m=n=2$, this occurs because $\mathbb{C}^{m,1}$ does not contain negative definite subspaces of codimension 2 while $\mathbb{C}^{2,1}$ does contain such spaces. This phenomenon makes the set $\Lambda$ become ``larger'' for $m>2$, in fact in this case the set $\Lambda$ has non-empty interior. In order to study the case $m<n$ we consider the projection 
\begin{gather*}
Q: \mathbb{P}_\mathbb{C}^n \setminus \Lambda_0 \rightarrow \mathbb{P}_\mathbb{C}^m \\
[z_1: \cdots :z_{n+1}] \mapsto [z_1: \cdots :z_m: z_{n+1}]
\end{gather*}
which is induced by the map 
\begin{gather} \label{eq:pro}
\widetilde{Q}:\mathbb{C}^{n,1} \rightarrow \mathbb{C}^{m,1} \\
(z_1, \cdots ,z_{n+1}) \mapsto (z_1, \cdots ,z_m, z_{n+1}) \notag
\end{gather}
the set $\Lambda_0$ is the projectivization of the kernel of the map $\widetilde{Q}$ and is contained in the Kulkarni limit set $\Lambda$, the projection $Q$ allows to obtain the topological information of the sets $\Lambda$ and $\Omega$ from the information obtained when $m$ and $n$ coincide. Actually, if we denote by $\Lambda_{(m)}$ the Kulkarni limit set of $\Gamma$ when it acts on the copy of $\mathbb{P}_\mathbb{C}^m$, then the set $\Lambda$ is obtained as the union of $\Lambda_0$ and the preimage under $Q$ of $\Lambda_{(m)}$. Also, writing $\Omega_{(m)}$ the complement of $\Lambda_{(m)}$ in $\mathbb{P}_\mathbb{C}^m$ we obtain that $\Omega$ is the preimage under $Q$ of $\Omega_{(m)}$. Our main results can be summarized in the following theorems. 
\begin{theorem}\label{teo:uno}
Let $\Gamma$ be a discrete subgroup of $\mathrm{SO}_+(m,1)$ acting in $\mathbb{H}_{\mathbb{R}}^m$, whose limit set is $\mathbb{S}^{m-1}$. If $\Omega$ is the Kulkarni discontinuity region of $\iota(\Gamma)$, see \eqref{eq:0}, acting in $\mathbb{P}_{\mathbb{C}}^m$, then there exists an $\mathrm{SO_+}(m,1)$-equivariant smooth fibre bundle $\Pi$ from $\Omega$ to $\mathbb{H}_{\mathbb{R}}^{m}$.   
\end{theorem}

\begin{theorem}\label{teo:dos}
Let $\Gamma$ be a discrete subgroup of $\mathrm{SO}_+(m,1)$ acting in $\mathbb{H}_{\mathbb{R}}^m$, whose limit set is $\mathbb{S}^{m-1}$. If $\Lambda$ is the Kulkarni limit set of $ \iota(\Gamma)$, see \eqref{eq:0}, acting in $\mathbb{P}_{\mathbb{C}}^m$, then $\Lambda$ is a semi-algebraic set. Moreover, if $m>2$, $\Lambda$ has non empty interior.   
\end{theorem}

\begin{theorem}\label{teo:tres}
Let $\Gamma$ be a discrete subgroup of $\mathrm{SO}_+(m,1)$ acting in $\mathbb{H}_{\mathbb{R}}^m$, whose limit set is $\mathbb{S}^{m-1}$. If $\Omega$ is the Kulkarni discontinuity region of $\iota(\Gamma)$, see \eqref{eq:0}, acting in $\mathbb{P}_{\mathbb{C}}^n$, where $n \geq m$, then 
\begin{itemize}
	\item[i)] $\Omega$ has three connected components, for $m=2$;
	\item[ii)] $\Omega$ is connected, for $m>2$. 
\end{itemize}
\end{theorem} 

The map $\Pi$ that appears in Theorem \ref{teo:uno} assigns to $z \in \Omega$ the class $[\mathbf{z} \overline{\eta} (\mathbf{z})+ \mathbf{\overline{z}} \eta (\mathbf{z})]$, where $\eta^2(\mathbf{z})\!=\!-<\!\mathbf{z},\mathbf{\overline{z}}\!>$, we can compute explicitly the fibre over the origin and use the equivariance of $\Pi$ to show that it is a smooth fibre bundle. Theorem \ref{teo:dos} follows from the Lemmas \ref{le:5} and \ref{le:6} which use the Sylvester's law of Inertia to give a criterion to decide whether or not a point is in $\Lambda$ for $m>2$, the result for $m=2$ was proved in \cite{cano2017limit}. Theorem \ref{teo:tres} follows from Lemmas \ref{le:17} and \ref{le:19} which are obtained using the projections defined in \eqref{eq:pro}. 
We have the following consequence of our results:
\begin{corollary}
Let $\Gamma$ be a discrete subgroup of $\mathrm{SO}_+(m,1)$ acting in $\mathbb{H}_{\mathbb{R}}^m$, whose limit set is $\mathbb{S}^{m-1}$. If $\Omega$ is the Kulkarni discontinuity region of $\iota(\Gamma)$, see \eqref{eq:0}, acting in $\mathbb{P}_{\mathbb{C}}^n$, where $n \geq m$, then the quotient $\Omega / \Gamma$ is the union of a finite number of $n$-dimensional complex manifolds.  
\end{corollary} 
The main interest in counting connected components comes from Kleinian groups in $\mathrm{PSL}(2,\mathbb{C})$ where it is known that the number of components of the discontinuity region can only be $0$, $1$, $2$ or infinity. For higher dimensional Kleinian groups there is no result on the number of components, our results provide examples where the Kulkarni discontinuity region is connected or has three components. In particular, we calculate the number of components for lattices in $\mathrm{SO}_+(m,1)$ for $m \geq 2$.   
The paper is organized as follows: Section \ref{sec2} is devoted to generalities about complex hyperbolic geometry and projective geometry. The case of a discrete subgroup of $\mathrm{SO}_+(m,1)$ acting on $\mathbb{P}_\mathbb{C}^m$ viewed as a subgroup of $\mathrm{SU}(m,1)$ via the embedding $\iota$ is discussed in Section \ref{sec3}. In Section \ref{sec4}, we treated the case when the subgroup $\Gamma$ of $\mathrm{SO}_+(2,1)$ has as limit set a $\mathbb{R}$-circle, and $n$ is greater than 2. Finally, the remaining cases are studied in Section \ref{sec5}.

\section{Preliminaries}\label{sec2}
   \subsection{Complex Projective Spaces and Complex Hyperbolic Geometry}
\label{subsec1}
We recall some results about complex hyperbolic geometry, general references for complex hyperbolic geometry are \cite{bridson2013metric}, \cite{kapovich2022survey},\cite{parker2003notes},\cite{parker2008hyperbolic} and \cite{schwartz2007spherical} . We call projective complex space to the quotient of $\mathbb{C}^{n+1} \setminus \{\mathbf{0}\}$ obtained by the equivalence relation: $\mathbf{u} \sim \mathbf{v}$ if there is $\alpha \in \mathbb{C}^*$ such that $\mathbf{u}= \alpha \mathbf{v}$. We denote by $\mathbb{P}_\mathbb{C}^n$ the projective complex space and by $[ \; \;]$ the quotient map that sends a point $\mathbf{z}=(z_1,z_2,...,z_{n+1})$ to its corresponding class, denoted $[z_1:z_2:...:z_{n+1}]$.\\
 
A complex projective  $(d-1)$-subspace of $\mathbb{P}_\mathbb{C}^n$ is a subset $W$ of $\mathbb{P}_\mathbb{C}^n$ such that $[W]^{-1}  \cup \{0\}$ is a complex linear $d$-subspace of $\mathbb{C}^{n+1}$. Given a subset $A \subset \mathbb{P}_\mathbb{C}^n$, we denote by $\langle A\rangle$ the smallest projective subspace containing $A$ and, we call it the projective subspace generated by $A$. For $\mathbf{z_1}$ and $\mathbf{z_2}$ linearly independent vectors in $\mathbb{C}^{n+1}$, we denote by $\overleftrightarrow{\mathbf{z_1},\mathbf{z_2}}$ the complex $2$-plane spanned by  $\mathbf{z_1}$ and $\mathbf{z_2}$. Also we denote by $\overleftrightarrow{z_1,z_2}$ to the line $\langle \{z_1,z_2\} \rangle$, where $z_1=[\mathbf{z_1}]$ and $z_2=[\mathbf{z_2}]$. The next lemma is the translation of the dimension formula for vector spaces to the setting of projective spaces, see \cite{coxeter2003projective}.
\begin{lemma}
Given $L$ and $L'$ two projective subspaces of $\mathbb{P}_\mathbb{C}^n$, the following formula holds: 
\begin{equation} \notag
\mathrm{dim}\:L+\mathrm{dim}\:L'=\mathrm{dim}(L \cap L')+\mathrm{dim}(\langle L \cup L' \rangle).
\end{equation}   
\end{lemma}  

As a consequence: If $\mathrm{dim}\:L+\mathrm{dim}\:L' \geq n$, the intersection $L \cap L'$ is non-empty. We denote by $\mathbb{C}^{n,1}$ the complex vector space $\mathbb{C}^{n+1}$ equipped with the non-degenerate indefinite Hermitian form of signature $(n,1)$:
\begin{equation}\notag 
\langle\mathbf{z},\mathbf{w}\rangle=z_1 \overline{w_1}+z_n \overline{w_2}+ \cdots z_n \overline{w_n}-z_{n+1} \overline{w_{n+1}},
\end{equation}  
where $\mathbf{z}=(z_1,z_2, \ldots,z_{n+1})$ and $\mathbf{w}=(w_1,w_2, \ldots,w_{n+1})$. Given two vectors $\mathbf{z},\mathbf{w}$ in $\mathbb{C}^{n+1}$, we say that 
$\mathbf{z}$ and $\mathbf{w}$ are orthogonal if $\langle\mathbf{z},\mathbf{w}\rangle=0$. Consider $W$ a subspace of $\mathbb{C}^{n+1}$, the set 
\begin{equation}\notag 
W^{\bot}=\{\mathbf{z} \in \mathbb{C}^{n+1} : \langle\mathbf{z},\mathbf{w}\rangle=0 \; \mathrm{for \; every \;\mathbf{w} \in W  } \}
\end{equation}
 is a subspace of $\mathbb{C}^{n+1}$ that we call the orthogonal complement of $W$. We use the same terminology as in \cite{chen1974hyperbolic}. That is, if $W$ is a subspace of $\mathbb{C}^{n+1}$, we say that $W$ is hyperbolic, elliptic or parabolic if the Hermitian form $H$ restricted to $W$ is respectively non-degenerate and indefinite, positive definite or degenerate. \\

When $W$ is a hyperbolic (respectively elliptic) subspace then $\mathbb{C}^{n+1}=W \oplus W^{\bot}$, so $W^{\bot}$ must be elliptic (respectively hyperbolic). If $W$ is a parabolic subspace, then so is $W^{\bot}$.  We use the following notation for the null, positive and negative sets of vectors in $\mathbb{C}^{n+1}$: 

\begin{align*} 
V_0=&\{\mathbf{z} \in \mathbb{C}^{n,1} \setminus \{\mathbf{0}\}: \langle\mathbf{z},\mathbf{z}\rangle=0\}, \\
V_+=&\{\mathbf{z} \in \mathbb{C}^{n,1}: \langle\mathbf{z},\mathbf{z}\rangle>0\}, \\
V_-=&\{\mathbf{z} \in \mathbb{C}^{n,1}: \langle\mathbf{z},\mathbf{z}\rangle<0\}.
\end{align*}

An automorphism that preserves the Hermitian form is called an unitary transformation and we denote the group of all unitary transformations by $
\mathrm{U}(n,1)$. The sets $V_0,V_+$ and $V_-$ are preserved by $\mathrm{U}(n,1)$, also $\mathrm{U}(n,1)$ acts transitively in $V_-$ and doubly transitively in $V_0$. The projectivization of $V_-$: 

\begin{equation}\notag
[V_-]=\{[z_1:\cdots :z_n:1] \in \mathbb{P}_{\mathbb{C}}^n: \lvert z_1 \rvert ^2 +\cdots  +\lvert z_n \rvert ^2< 1 \}
\end{equation}

is a complex $n$-dimensional open ball in $\mathbb{P}_{\mathbb{C}}^n$. The set $[V_-]$ equipped with the quadratic form induced by the Hermitian Form $\langle \cdot \; , \cdot \rangle$ is a model for the complex hyperbolic space $\mathbb{H}_{\mathbb{C}}^n$. In the same manner, we obtain that $[V_0]$ is the $(2n-1)$-sphere in $\mathbb{P}_{\mathbb{C}}^n$ that is the boundary of $\mathbb{H}_{\mathbb{C}}^n$. Finally, we obtain that $[V_+]$ is the complement of the complex n-dimensional closed ball $\overline{\mathbb{H}_{\mathbb{C}}^n}$.\\ 

The projectivization in $\mathrm{PGL}(n+1,\mathbb{C})$ of the unitary group $\mathrm{U}(n,1)$, denoted by $\mathrm{PU}(n,1)$, acts transitively in $\mathbb{H}_{\mathbb{C}}^n$ and by diffeomorphisms in the boundary $\partial \mathbb{H}_{\mathbb{C}}^n$.

\subsection{Sylvester's law of Inertia}
\label{subsec2}

In order to obtain analogous results to those in Cano et al. \cite{cano2016action},  we use the following Lemma known as Sylvester's law of Inertia.

\begin{lemma} \label{le:sy}
Let $\mathbf{z_1}$ and $\mathbf{z_2}$ be two linearly independent vectors in $\mathbb{C}^{n,1}$. Consider the matrix $G=\left(\begin{array}{cc}
\langle\mathbf{z_1},\mathbf{z_1}\rangle & \langle\mathbf{z_1},\mathbf{z_2}\rangle\\
\langle\mathbf{z_2},\mathbf{z_1}\rangle & \langle\mathbf{z_2},\mathbf{z_2}\rangle
\end{array}\right)
 $, we have that $\overleftrightarrow{\mathbf{z_1},\mathbf{z_2}}$ is elliptic, hyperbolic, or parabolic if and only if $G$ has two positive eigenvalues, one negative eigenvalue, or one zero eigenvalue, respectively. Moreover, $\overleftrightarrow{\mathbf{z_1},\mathbf{z_2}}$ is elliptic, hyperbolic, or parabolic if and only if the determinant of $G$ is positive, negative, or zero, respectively.
\end{lemma}
 We give the proof of the elliptic case. The other cases use similar arguments, and we omit their proof:
\begin{proof}
The eigenvalues of $G$ are \begin{align*}
\lambda_1 &=\frac{\langle \mathbf{z_1},\mathbf{z_1}\rangle+\langle \mathbf{z_2},\mathbf{z_2}\rangle + \sqrt{(\langle \mathbf{z_1},\mathbf{z_1}\rangle-\langle \mathbf{z_2},\mathbf{z_2}\rangle)^2+4 \lvert \langle \mathbf{z_1},\mathbf{z_2}\rangle \rvert^2}}{2}, \; \mathrm{and} \\
\lambda_2 &=\frac{\langle \mathbf{z_1},\mathbf{z_1}\rangle+\langle \mathbf{z_2},\mathbf{z_2}\rangle - \sqrt{(\langle \mathbf{z_1},\mathbf{z_1}\rangle-\langle \mathbf{z_2},\mathbf{z_2}\rangle)^2+4 \lvert \langle \mathbf{z_1},\mathbf{z_2}\rangle \rvert^2}}{2}.
\end{align*}  
If $\overleftrightarrow{\mathbf{z_1},\mathbf{z_2}}$ is elliptic then $\lambda_1>0$ since $\langle\mathbf{z_1},\mathbf{z_1}\rangle,\langle\mathbf{z_2},\mathbf{z_2}\rangle>0$. Also, 
\begin{align*}\left\langle\mathbf{z_1}- \mathbf{z_2} \frac {\langle\mathbf{z_1},\mathbf{z_2}\rangle} {\langle\mathbf{z_2},\mathbf{z_2}\rangle},\mathbf{z_1}- \mathbf{z_2} \frac{\langle\mathbf{z_1},\mathbf{z_2}\rangle} {\langle\mathbf{z_2},\mathbf{z_2}\rangle}\right\rangle=\frac{\langle\mathbf{z_1},\mathbf{z_1}\rangle \langle\mathbf{z_2},\mathbf{z_2}\rangle -\lvert \langle\mathbf{z_1},\mathbf{z_2}\rangle \rvert ^2}{\langle\mathbf{z_2},\mathbf{z_2}\rangle}>0.
\end{align*} Thus both eigenvalues and the determinant of $G$ are positive.  In fact, the inequality $\langle\mathbf{z_1},\mathbf{z_1}\rangle\langle\mathbf{z_2},\mathbf{z_2}\rangle -\lvert \langle\mathbf{z_1},\mathbf{z_2}\rangle \rvert ^2>0$ is  the Cauchy's inequality for $\mathbf{z_1}$ and $\mathbf{z_2}$, and this inequality holds for elliptic spaces. Conversely, if $\mathrm{det}(G)=\langle\mathbf{z_1},\mathbf{z_1}\rangle \langle\mathbf{z_2},\mathbf{z_2}\rangle -\lvert \langle\mathbf{z_1},\mathbf{z_2}\rangle \rvert ^2>0$, then the eigenvalues are both negative or both positive. In the first case, $G$ is similar to $-I$, thus $\overleftrightarrow{\mathbf{z_1},\mathbf{z_2}}$ is  negative definite, a contradiction. So, both eigenvalues are positive, $G$ is similar to $I$ and  $\overleftrightarrow{\mathbf{z_1},\mathbf{z_2}}$ is positive definite.       
\end{proof}

\section{A real $(n-1)$-sphere in $\mathbb{P}_{\mathbb{C}}^{n}$}
\label{sec3}

We extend the work of Cano et al. \cite{cano2016action}, for $n>2$, in the sense that given $S$ a $\mathbb{R}-$sphere of highest dimension in $\mathbb{P}_{\mathbb{C}}^{n}$, we describe the set $\Lambda$ defined as the union of all complex hyperplanes tangent to $\partial\mathbb{H}_{\mathbb{C}}^{n}$ at points in $S$, also we show that $\Omega:=\mathbb{P}_{\mathbb{C}}^{n} \setminus \Lambda$ is connected. 

Consider the real $(n-1)$-sphere 
\begin{equation}\notag\partial \mathbb{H}_{\mathbb{R}}^n=\partial \mathbb{H}_{\mathbb{C}}^n \cap  \mathbb{P}_{\mathbb{R}}^n,
\end{equation} 
where $\mathbb{P}_{\mathbb{R}}^n=\{[r_1:r_2:\cdots : r_{n+1}] \in \mathbb{P}_{\mathbb{C}}^n: r_j \in \mathbb{R}, \; 1 \leq j \leq n+1 \}$. It is a well known fact that all the other real $(n-1)$-spheres are obtained by the usual action of $\mathrm{PU}(n,1)$ on $\partial \mathbb{H}_{\mathbb{C}}^n$, see \cite{goldman1999complex}.  \\
Given $p \in \mathbb{P}_{\mathbb{C}}^n$, we denote by $H_p$ the projective hyperplane 
\begin{equation}\notag
\{w \in \mathbb{P}_{\mathbb{C}}^n : p_1 \overline{w_1} + \cdots +p_n \overline{w_n} -p_{n+1} \overline{w_{n+1}}=0\}.
\end{equation}
 Then we can write the set $\Lambda$ as the set: 
\begin{equation}\notag
\underset{p\in\partial\mathbb{H}_{\mathbb{R}}^{n}}{\bigcup}H_{p}.
\end{equation}

The next Lemma is just a restatement in the $n$-dimensional case of the part $(i)$ of Proposition 4.1 in Barrera et al. \cite{barrera2019chains}.

\begin{lemma} \label{le:la} 
The set $\Lambda$ can be written as:
\begin{equation}\notag
\{z \in \mathbb{P}_{\mathbb{C}}^n:H_z \cap \partial\mathbb{H}_{\mathbb{R}}^{n} \neq \emptyset\}.
\end{equation}
\end{lemma}

\begin{proof}
Consider $z \in \mathbb{P}_{\mathbb{C}}^n$, $z$ lies in $\Lambda$ if and only if there is $w$ in $\partial\mathbb{H}_{\mathbb{R}}^{n}$ such that $z$ lies in  $H_w$, that is, if and only if $w$ lies in  $H_z$.
\end{proof}

We use this characterization to determine when the projectivization of a vector $\mathbf{z}=(z_1,\ldots,z_{n+1}) \in \mathbb{C}^{n,1} \setminus \{\mathbf{0}\}$ lies in $\Lambda$. 

\begin{lemma} \label{le:1}
Let $\mathbf{z}$ be a vector in $\mathbb{C}^{n,1} \setminus \{\mathbf{0}\}$. The following sentences are equivalent: 
\begin{enumerate}
	\item The vectors $\mathbf{z}$ and $\overline{\mathbf{z}}$ are linearly dependent, 
	\item  the complex numbers $z_j \overline{z_k}-z_k \overline{z_j}$ vanish for $1 \leq j,k \leq n+1$,
	\item the point $z=[\mathbf{z}]$, lies in $\mathbb{P}_{\mathbb{R}}^n$.
\end{enumerate}
 \end{lemma}

\begin{proof}
 The third statement follows from the first one, since if $\mathbf{z}$ and $\overline{\mathbf{z}}$ are linearly dependent, then $z$ is invariant under complex conjugation, that is, $z$ belongs to $\mathbb{P}_{\mathbb{R}}^n$.\\

Now, if the third statement holds, then there are a $\mathbf{w}$ in $\mathbb{R}^{n,1}$ and an $\alpha \in \mathbb{C}^* $ such that $\mathbf{z}=\alpha \mathbf{w}$, so $z_j \overline{z_k}-z_k \overline{z_j}$ is equal to $\lvert \alpha \rvert^2 (w_j w_k-w_k w_j)=0$ for every $1 \leq j,k \leq n+1$.\\

Finally, we prove that the second statement implies the linear dependence of $\mathbf{z}$ and $\overline{\mathbf{z}}$. Without loss of generality, we suppose that $z_1 \neq 0$, so $\overline{z_k}=z_k \frac{\overline{z_1}}{z_1}$ for $1 \leq k \leq n+1$; in other words, there is $\alpha= \frac{\overline{z_1}}{z_1}$ such that $\mathbf{\overline{z}}=\alpha \mathbf{z}$.\\ 
\end{proof}

Lemmas \ref{le:la} and \ref{le:1} imply that:

\begin{proposition} \label{pro:1}
A point $z=[\mathbf{z}]$ lies in $\Lambda$ if and only if there is  $\mathbf{w}\in\mathbb{C}^{n,1}$ satisfying the following three conditions:
\begin{enumerate}
	\item  $w=[\mathbf{w}]\in H_z$, 
	\item $\mathbf{w}$ and $\overline{\mathbf{w}}$ are linearly dependent, and 
	\item $\langle\mathbf{w},\mathbf{w}\rangle=0$.
\end{enumerate}
\end{proposition}

The first two conditions above give a system of equations that can be interpreted geometrically. Actually, we can write the condition 1 as, 
	\begin{equation} \label{eq:1}
	z_1\overline{w_1}+ \cdots z_n\overline{w_n}-z_{n+1} \overline{w_{n+1}} =0,
	\end{equation}
 and using the condition 2, we can rewrite it in the form,  

\begin{equation} \label{eq:2}
	z_1 w_1+ \cdots z_n w_n-z_{n+1} w_{n+1} =0,
	\end{equation}

additionally, taking the conjugate of Equation \eqref{eq:1} we obtain,

\begin{equation}  \label{eq:3*}
	\overline{z_1} w_1+ \cdots \overline{z_n} w_n- \overline{z_{n+1}} w_{n+1} =0.
	\end{equation}  

Equations \eqref{eq:2} and \eqref{eq:3*} imply that $\mathbf{w}$ belongs to the orthogonal hyperplane to $\mathbf{z}$ and to the orthogonal hyperplane to $\overline{\mathbf{z}}$, respectively. If we denote by $\mathbf{H_p}$ the hyperplane orthogonal to a vector $\mathbf{p}$ in $\mathbb{C}^{n,1}$, then $\mathbf{w}$ belongs to the intersection $\mathbf{H_z} \cap \mathbf{H}_{\overline{\mathbf{z}}}$.  

\begin{proposition}
There are only two possibilities for the intersection $\mathbf{H_z} \cap \mathbf{H}_{\overline{\mathbf{z}}}$:   
\begin{itemize}
	\item[a)] The intersection $\mathbf{H_z} \cap \mathbf{H}_{\overline{\mathbf{z}}}$ is the hyperplane $\mathbf{H_z}$. It occurs if and only if $\mathbf{z}$ and $\overline{\mathbf{z}}$ are linearly dependent;
	\item[b)] the intersection $\mathbf{H_z} \cap \mathbf{H}_{\overline{\mathbf{z}}}$ is a codimension 2 subspace, and it happens if and only if $\mathbf{z}$ and $\overline{\mathbf{z}}$ are linearly independent.  
\end{itemize}
\end{proposition}

In either case, the intersection $\mathbf{H_z} \cap \mathbf{H}_{\overline{\mathbf{z}}}$ is preserved by conjugation, so, $\mathbf{H_z} \cap \mathbf{H}_{\overline{\mathbf{z}}}$ is either the complexification of a hyperplane in $\mathbb{R}^{n,1}$ or the complexification of a codimension 2 subspace in $\mathbb{R}^{n,1}$.

\subsection{Case 1: $\mathbf{z}$ and $\overline{\mathbf{z}}$ linearly independent}
\label{subsec3}

We first note that the plane $\overleftrightarrow{\mathbf{z},\mathbf{\overline{z}}}$ spanned by $\mathbf{z}$ and $\mathbf{\overline{z}}$ is orthogonal to $\mathbf{H_z} \cap \mathbf{H}_{\overline{\mathbf{z}}}$. Moreover, $\overleftrightarrow{\mathbf{z},\mathbf{\overline{z}}}$ is the orthogonal complement of $\mathbf{H_z} \cap \mathbf{H}_{\overline{\mathbf{z}}}$. 

\begin{lemma} \label{le:phe}
Let $\mathbf{z}$ be a vector in $\mathbb{C}^{n,1}$ such that $\mathbf{z}$ and $\overline{\mathbf{z}}$ are linearly independent. The following statements hold:
\begin{itemize}
\item [a)] The set $\mathbf{H_z} \cap \mathbf{H}_{\overline{\mathbf{z}}} \setminus \{\mathbf{0}\}$ is contained in $V_+$ if and only if $\overleftrightarrow{\mathbf{z},\mathbf{\overline{z}}}$ has a negative vector and $(\mathbf{H_z} \cap \mathbf{H}_{\overline{\mathbf{z}}}) \cap \overleftrightarrow{\mathbf{z},\mathbf{\overline{z}}}=\{\mathbf{0}\}$;
\item [b)] the set $\overleftrightarrow{\mathbf{z},\mathbf{\overline{z}}} \setminus\{\mathbf{0}\}$  is contained in $V_+$ if and only if $\mathbf{H_z} \cap \mathbf{H}_{\overline{\mathbf{z}}}$ has a negative vector and  $(\mathbf{H_z} \cap \mathbf{H}_{\overline{\mathbf{z}}}) \cap \overleftrightarrow{\mathbf{z},\mathbf{\overline{z}}}=\{\mathbf{0}\}$;
\item [c)] the space $\mathbf{H_z} \cap \mathbf{H}_{\overline{\mathbf{z}}}$ is parabolic if and only if $\overleftrightarrow{\mathbf{z},\mathbf{\overline{z}}}$ is parabolic.
\end{itemize} 
\end{lemma}
\begin{proof}
This is a particular case of the well known fact that the orthogonal complement of a hyperbolic, elliptic or parabolic subspace is (respectively) elliptic, hyperbolic or parabolic; however, we provide a proof.\\

We first note that the subspaces $\mathbf{H_z} \cap \mathbf{H}_{\overline{\mathbf{z}}}$ and $\overleftrightarrow{\mathbf{z},\mathbf{\overline{z}}}$ are not parabolic if and only if the intersection $(\mathbf{H_z} \cap \mathbf{H}_{\overline{\mathbf{z}}}) \cap \overleftrightarrow{\mathbf{z},\mathbf{\overline{z}}}$ is $\{\mathbf{0}\}$, and it happens, if and only if $\mathbb{C}^{n,1}$ is decomposed as the direct sum of $\mathbf{H_z} \cap \mathbf{H}_{\overline{\mathbf{z}}}$ and $\overleftrightarrow{\mathbf{z},\mathbf{\overline{z}}}$.\\
a) Given that $\mathbf{H_z} \cap \mathbf{H}_{\overline{\mathbf{z}}} \setminus \{\mathbf{0}\} \subset V_+$ then $\mathbb{C}^{n,1}$ is the direct sum of $\mathbf{H_z} \cap \mathbf{H}_{\overline{\mathbf{z}}}$ and $\overleftrightarrow{\mathbf{z},\mathbf{\overline{z}}}$, thus $\overleftrightarrow{\mathbf{z},\mathbf{\overline{z}}}$ has a negative vector. Conversely, $\mathbb{C}^{n,1}= (\mathbf{H_z} \cap \mathbf{H}_{\overline{\mathbf{z}}}) \oplus \overleftrightarrow{\mathbf{z},\mathbf{\overline{z}}}$ and $\overleftrightarrow{\mathbf{z},\mathbf{\overline{z}}}$ has a negative vector, since the total space has signature $(n,1)$ then  $\mathbf{H_z} \cap \mathbf{H}_{\overline{\mathbf{z}}} \setminus \{\mathbf{0}\} \subset V_+$. \\
b) It follows by the same arguments used in a).\\
c) $\mathbf{H_z} \cap \mathbf{H}_{\overline{\mathbf{z}}}$ is parabolic if and only if $(\mathbf{H_z} \cap \mathbf{H}_{\overline{\mathbf{z}}}) \cap \overleftrightarrow{\mathbf{z},\mathbf{\overline{z}}} \neq \{\mathbf{0}\}$, and it happens if and only if $\overleftrightarrow{\mathbf{z},\mathbf{\overline{z}}}$ is parabolic. 
 \end{proof}

\begin{corollary}
 If the plane $\overleftrightarrow{\mathbf{z},\mathbf{\overline{z}}}$ is elliptic or parabolic then, $\mathbf{z}$ is positive. 
\end{corollary}

Now, we generalize the function $f$ given by Cano et al. \cite{cano2016action} by means of Sylvester's law of inertia. In order to define $f$, we note that the determinant of $G=\left( \begin{array}{cc}
\langle\mathbf{z},\mathbf{z}\rangle & \langle\mathbf{z},\overline{\mathbf{z}}\rangle\\
\langle\overline{\mathbf{z}},\mathbf{z}\rangle & \langle\overline{\mathbf{z}},\overline{\mathbf{z}}\rangle
\end{array} \right)$  is:
 \begin{align*}
\langle\mathbf{z},\mathbf{z}\rangle^2 -\lvert \langle\mathbf{z},\mathbf{\overline{z}}\rangle \rvert ^2={\displaystyle \sum_{j=1}^{n}}(z_{j}\overline{z_{n+1}}-z_{n+1}\overline{z_{j}})^{2}-\sum_{1\leq j<k \leq n}(z_{j}\overline{z_{k}}-z_{k}\overline{z_{j}})^{2}\\
={\displaystyle -\sum_{j=1}^{n}}4(x_jy_{n+1}-x_{n+1}y_{j})^{2}+\sum_{1\leq j<k \leq n} 4(x_jy_{k}-y_{k}x_{j})^{2}.
 \end{align*} 

Where $x_j=\mathrm{Re}(z_j)$ and $y_j=\mathrm{Im}(z_j)$. Thus, we define $f$ as follows: 
\begin{equation}\notag
f: \mathbb{C}^{n,1} \setminus \{\mathbf{0}\} \rightarrow \mathbb{R}  
\end{equation}
 \begin{equation} \label{eq:3} \notag
f(\mathbf{z})={\displaystyle \sum_{j=1}^{n}}4(x_jy_{n+1}-x_{n+1}y_{j})^{2}-\sum_{1\leq j<k \leq n} 4(x_jy_{k}-y_{k}x_{j})^{2}.
\end{equation}

The function $f$ allow us to characterize the set $\Lambda$ when 
$\mathbf{z}$ and $\mathbf{\overline{z}}$ are linearly independent, and $n>2$. 

\begin{lemma} \label{le:5} 
 Let $n>2$ be an integer. Consider $\mathbf{z}$ a vector in $\mathbb{C}^{n,1}$ such that $\mathbf{z}$ and $\mathbf{\overline{z}}$ are linearly independent. The projectivization $z=[\mathbf{z}]$ lies in $\Lambda$ if and only if $f(\mathbf{z}) \leq 0$.   
\end{lemma}
\begin{proof}
By the remark just below Proposition \ref{pro:1}, $z \in \Lambda$ if and only if $\mathbf{H_z} \cap \mathbf{H}_{\overline{\mathbf{z}}}$ is hyperbolic or parabolic, and by Lemma \ref{le:phe} it happens if and only if the plane $\overleftrightarrow{\mathbf{z},\mathbf{\overline{z}}}$ is elliptic or parabolic, and by Lemma \ref{le:sy}, the latter occurs, if and only if $f(\mathbf{z}) \leq 0$.  
\end{proof}

The Lemma \ref{le:5} exhibits a difference between the cases $n=2$ and $n>2$. We recall that if $n=2$, see Corollary 2.5 in Cano et al. \cite{cano2016action}, it is necessary that $f(\mathbf{z}) = 0$ to ensure that $z=[\mathbf{z}]$ lies in $\Lambda$. This behavior is due to that for $n=2$, $\mathbf{H_z} \cap \mathbf{H}_{\overline{\mathbf{z}}}$ is the complex line spanned by $i\mathbf{z} \boxtimes \mathbf{\overline{z}}$ and this line can be contained in $V_-$. In other words, $\mathbf{H_z} \cap \mathbf{H}_{\overline{\mathbf{z}}}$ can be a negative definite subspace because  it has dimension $1$. However in the case $n>2$, $\mathbf{H_z} \cap \mathbf{H}_{\overline{\mathbf{z}}}$ can not be a negative definite subspace because it has codimension $2$ and thus dimension greater than $1$.  
The case $n=2$ is a threshold at which the set $\Lambda$  ``grows'', in fact for $n>2$ it can be observed from inequality $f(\mathbf{z}) \leq 0$ that $\Lambda$ has non-empty interior.  

\subsection{Case 2: $\mathbf{z}$ and $\overline{\mathbf{z}}$ linearly dependent}
\label{subsec4}

If $\mathbf{z} \neq \mathbf{0}$, and satisfies that $\mathbf{z}$ and $\overline{\mathbf{z}}$ are linearly dependent, then by Lemma \ref{le:1}, $z=[\mathbf{z}] \in \mathbb{P}_{\mathbb{R}}^n$ and $f(\mathbf{z})=0$.  

\begin{lemma} \label{le:6}
Let $\mathbf{z}$ be a vector in $\mathbb{C}^{n,1}\setminus \{\mathbf{0}\}$ such that $\mathbf{z}$ and $\mathbf{\overline{z}}$ are linearly dependent. The projectivization $z=[\mathbf{z}]$ lies in $\Lambda$ if and only if $\langle\mathbf{z},\mathbf{z}\rangle \geq 0$.
\end{lemma}

\begin{proof}
The necessary condition is straightforward: if $z \in \Lambda$ then $\mathbf{z}$ is non-negative. Now, if $\mathbf{z}$ is null then $z \in \partial \mathbb{H}_{\mathbb{R}}^n \subset \Lambda$ because $z \in \mathbb{P}_{\mathbb{R}}^n$. If $\mathbf{z}$ is positive, then $z \in \mathbb{P}_{\mathbb{R}}^n \setminus \overline{\mathbb{H}_{\mathbb{R}}^n}$,  so there is a real projective hyperplane $H $ tangent to $\partial \mathbb{H}_{\mathbb{R}}^n$ that passes through $z$. Thus $z \in \Lambda$ because the complexification of $H$ is a projective hyperplane contained in $\Lambda$.  
\end{proof}

Note that Lemma \ref{le:6} does not depend of $n$ but Lemma \ref{le:5} does.  

\begin{proof} [Proof of Theorem \ref{teo:dos}] \label{proo:1}
For $n>2$ the Lemmas \ref{le:5} and \ref{le:6} imply that the set $\Lambda$ is the semi-algebraic subset of $\mathbb{P}_{\mathbb{C}}^n$ consisting of points $z=[\mathbf{z}]$ that satisfy the following pair of inequalities: 
\begin{align*}
1) &\; 0 \leq \langle\mathbf{z},\mathbf{z}\rangle= x_1^2+y_1^2+ \cdots +x_n^2+y_n^2-x_{n+1}^2-y_{n+1}^2 \\
2) & \; 0 \geq f(\mathbf{z})= {\displaystyle \sum_{j=1}^{n}}4(x_jy_{n+1}-x_{n+1}y_{j})^{2}-\sum_{1\leq j<k \leq n} 4(x_jy_{k}-y_{k}x_{j})^{2}.
\end{align*}
Thus $\Lambda$ has non-empty interior. If $n=2$, the result is the Corollary 2.5 in \cite{cano2016action}. In this case inequality 2 is indeed an equality.  
\end{proof} 

\begin{lemma} \label{le:7}
The function $f(\mathbf{z})$ is invariant under the usual action of $\mathrm{SO_+}(n,1)$
\end{lemma}

\begin{proof}
\begin{align*}
f(\mathbf{z}) & ={\displaystyle \sum_{j=1}^{n}}4(x_jy_{n+1}-x_{n+1}y_{j})^{2}-\sum_{1\leq j<k \leq n} 4(x_jy_{k}-y_{k}x_{j})^{2}\\
& =4\langle\mathbf{x},\mathbf{y}\rangle^2-4\langle\mathbf{x},\mathbf{x}\rangle \langle\mathbf{y},\mathbf{y}\rangle,
\end{align*}
where $\mathbf{x}=\mathrm{Re}(\mathbf{z})$ and $\mathbf{y}=\mathrm{Im}(\mathbf{z})$. So, if $A$ lies in $\mathrm{SO_+}(n,1)$, then $A \mathbf{z}=A\mathbf{x} + i A\mathbf{y}$, and   
\begin{align*}
f(A\mathbf{z}) &=4\langle A\mathbf{x},A\mathbf{y}\rangle^2-4\langle A\mathbf{x},A\mathbf{x}\rangle \langle A\mathbf{y},A\mathbf{y}\rangle\\
&=4\langle\mathbf{x},\mathbf{y}\rangle^2-4\langle\mathbf{x},\mathbf{x}\rangle \langle\mathbf{y},\mathbf{y}\rangle\\
&=f(\mathbf{z}),
\end{align*}
in the second equality we used that $A$ preserves the Hermitian form $\langle \cdot, \cdot \rangle$.   
\end{proof}

\subsection{A partition of $\mathbb{P}_{\mathbb{C}}^{n}$}
\label{subsec5} 
In this subsection we will assume that $n>2$. If the function 
\begin{equation}\notag 
f(\mathbf{z})=\displaystyle -\sum_{j=1}^{n}(z_j\overline{z_{n+1}}-z_{n+1} \overline{z_{j}})^{2}+\sum_{1\leq j<k \leq n}(z_j \overline{z_{k}}-z_{k} \overline{z_{j}})^2
\end{equation}
 takes the value $0$ then there are two possibilities: the complex numbers $z_j \overline{z_{k}}-z_{k} \overline{z_{j}}$ vanish for all $j$ and $k$, so $\mathbf{z}$ and $\mathbf{\overline{z}}$ are linearly dependent; or $z_j \overline{z_{k}}-z_{k} \overline{z_{j}}$ is no zero for some $j$ and $k$, and therefore $\mathbf{z}$ and $\mathbf{\overline{z}}$ are linearly independent.\\
 The function $f$ gives a partition of $\mathbb{C}^{n,1} \setminus \{\mathbf{0}\}$ in the following sets:
 \begin{align*} 
 U_0 & =\{\mathbf{z} \in \mathbb{C}^{n,1} \setminus \{\mathbf{0}\}: f(\mathbf{z})=0\} \\
=& \{\mathbf{z} \in \mathbb{C}^{n,1} \setminus \{\mathbf{0}\} : \mathbf{z} \; \mathrm{and} \; \mathbf{\overline{z}} \; \mathrm{are\; linearly\; dependent}\} \cup \{ \mathbf{z} \in \mathbb{C}^{n,1} \setminus \{\mathbf{0}\} : \overleftrightarrow{\mathbf{z},\mathbf{\overline{z}}} \; \mathrm{is \; parabolic}\}, \\
 U_+ & =\{\mathbf{z} \in \mathbb{C}^{n,1}: f(\mathbf{z})>0\}=\{ \mathbf{z} \in \mathbb{C}^{n,1}: \overleftrightarrow{\mathbf{z},\mathbf{\overline{z}}} \; \mathrm{is \; hyperbolic}\}, \\
 U_- & =\{\mathbf{z} \in \mathbb{C}^{n,1}: f(\mathbf{z})<0\}=\{ \mathbf{z} \in \mathbb{C}^{n,1}: \overleftrightarrow{\mathbf{z},\mathbf{\overline{z}}} \; \mathrm{is \; elliptic}\}.
\end{align*}
 
\begin{proposition}
The sets $[U_0],[U_+]$ and $[U_-]$ give a partition of $\mathbb{P}_\mathbb{C}^{n}$.
\end{proposition}
\begin{proof}
We remember that  $f(\mathbf{z})$ is equal to $\lvert \langle\mathbf{z},\mathbf{\overline{z}}\rangle \rvert ^2 - \langle\mathbf{z},\mathbf{z}\rangle^2$. So, if $\alpha$ lies in $\mathbb{C}^*$ then  
\begin{align*}
f(\alpha \mathbf{z}) &= \lvert \langle\alpha \mathbf{z},\overline{\alpha} \mathbf{\overline{z}}\rangle \rvert ^2 - \langle \alpha \mathbf{z},\alpha \mathbf{z}\rangle^2\\
& = \lvert \alpha \rvert ^4 (\lvert \langle\mathbf{z},\mathbf{\overline{z}}\rangle \rvert ^2 - \langle\mathbf{z},\mathbf{z}\rangle^2) \\
& = \lvert \alpha \rvert ^4 f(\mathbf{z}).
\end{align*}
Thus, $[U_0],[U_+]$ and $[U_-]$ are non-empty pairwise disjoint sets whose union is all $\mathbb{P}_\mathbb{C}^{n}$. 
\end{proof}

The next proposition is an analogous but different result to Theorem 2.7 in \cite{cano2016action}, in such proposition we can appreciate the ``growth'' of $\Lambda$.
\begin{proposition} \label{te:de}
The projective sets $[U_0],[U_+]$ and $[U_-]$ are $\mathrm{SO_+}(n,1)$-invariants and have the following properties:  

\begin{itemize}
	\item[i)] The union $([U_0] \setminus \mathbb{H}_{\mathbb{R}}^n) \cup [U_-]$ is equal to $\Lambda$;
	\item [ii)] The union $[U_+] \cup  \mathbb{H}_{\mathbb{R}}^n$ is equal to $\Omega=\mathbb{P}_{\mathbb{C}}^{n} \setminus \Lambda$;
	\item[iii)] The set $\Omega=\mathbb{P}_{\mathbb{C}}^{n} \setminus \Lambda$ is a complete Kobayashi hyperbolic space.  
\end{itemize}
\end{proposition}

\begin{proof}
By Lemma \ref{le:7} we know that the sets $[U_0],[U_+]$ and $[U_-]$ are $\mathrm{SO_+}(n,1)$-invariants. The parts i) and ii) follow from the description of $\Lambda$ given in the proof of Theorem 2 because when $\mathbf{z}$ lies in $U_0 \setminus \mathbb{R}^{n,1}$ or in $U_-$ then $\mathbf{z}$ is positive. \\

iii) By Corollary 3.10.9 in \cite{kobayashi2013hyperbolic} we know that an open subset of $\mathbb{P}_{\mathbb{C}}^n$ that omits at least $2n+1$ projective hyperplanes in general position is a complete Kobayashi hyperbolic space. Given that for any point in $\mathbb{P}_{\mathbb{R}}^n \setminus \overline{\mathbb{H}_{\mathbb{R}}^n}$ there are infinitely many projective hyperplanes in general position passing through the point and contained in $\Lambda$, then $\Lambda$ contains infinitely many projective hyperplanes in general position. Thus, $\Omega$ is a complete Kobayashi hyperbolic space.   
\end{proof}

\subsection{The fiber bundle $\Omega \rightarrow \mathbb{H}_{\mathbb{R}}^n$}
\label{subsec6}

Following Cano et al. \cite{cano2016action} we use the orthogonal projection $\Pi_{\mathbb{R}}:  \mathbb{H}_{\mathbb{C}}^n \rightarrow \mathbb{H}_{\mathbb{R}}^n$ that sends $z$ to its closest point in $\mathbb{H}_{\mathbb{R}}^n$. Goldman in \cite{goldman1999complex} proves that $\Pi_{\mathbb{R}}(z)$ is the midpoint $m(z)$ of the geodesic segment joining $z$ and $\overline{z}$. Explicitly, $m(z)=[\mathbf{m}(\mathbf{z})]$, where 
\begin{equation} \notag 
\mathbf{m}(\mathbf{z})=\mathbf{z} \overline{\eta} (\mathbf{z})+ \mathbf{\overline{z}} \eta (\mathbf{z})
\end{equation} 
and $\eta(\mathbf{z})$ is a number such that 
\begin{equation} \notag
\eta^2(\mathbf{z})=-\langle\mathbf{z},\mathbf{\overline{z}}\rangle=z_{n+1}^2-z_1^2-\cdots -z_n^2.
\end{equation}

For the remainder of this subsection we will assume that $n>2$. We know, by definition of $U_+$, that the space of non-real negative vectors is contained in $U_+$, so we define the function $\widetilde{\Pi}: U_+ \rightarrow \mathbb{R}^{n,1}$ given by 

\begin{equation} \label{eq:pi}
\widetilde{\Pi}(\mathbf{z})=\mathbf{z} \overline{\eta} (\mathbf{z})+ \mathbf{\overline{z}} \eta (\mathbf{z}).
\end{equation}

\begin{lemma} \label{le:8}
The function $\widetilde{\Pi}$ is $\mathrm{SO_+}(n,1)$-equivariant.
\end{lemma}

\begin{proof}
Since $A$ lies in $\mathrm{SO_+}(n,1)$, then $\overline{A\mathbf{z}}=A\mathbf{\overline{z}}$. So, we have that

\begin{equation} \notag
\eta^{2}(A\mathbf{z})=-\langle A\mathbf{z},\overline{A\mathbf{z}}\rangle=-\langle A\mathbf{z},A\mathbf{\overline{z}}\rangle=-\langle \mathbf{z},\mathbf{\overline{z}}\rangle=\eta^{2}(\mathbf{z}).
\end{equation}
Hence $\eta(A\mathbf{z})=\eta(\mathbf{z})$ and
\begin{equation}\notag
\widetilde{\Pi}(A\mathbf{z})=A\mathbf{z}\overline{\eta}(A\mathbf{z})+\overline{A\mathbf{z}}\eta(A\mathbf{z})=A\mathbf{z}\overline{\eta}(\mathbf{z})+A\overline{\mathbf{z}}\eta(\mathbf{z})=A(\mathbf{z}\overline{\eta}(\mathbf{z})+\overline{\mathbf{z}}\eta(\mathbf{z}))=A\widetilde{\Pi}(\mathbf{z}).
\end{equation}
That is $\widetilde{\Pi}$ is $\mathrm{SO_+}(n,1)$-equivariant.
\end{proof}

\begin{lemma}\label{le:9}
Consider the function $\widetilde{\Pi}$ defined by Equation \eqref{eq:pi}. The equality $\widetilde{\Pi}(\alpha \mathbf{z})=\lvert \alpha \rvert ^2 \widetilde{\Pi}(\mathbf{z})$ holds for every $\alpha \in \mathbb{C}^{*}$. 
\end{lemma}

\begin{proof}
Take $\alpha \in \mathbb{C}^{*}$, since
\begin{equation} \notag
\eta^{2}(\alpha \mathbf{z})=-\langle\alpha \mathbf{z},\overline{\alpha \mathbf{z}}\rangle=-\lvert\alpha\rvert^2 \langle\mathbf{z},\mathbf{\overline{z}}\rangle=\lvert\alpha\rvert^2 \eta^{2}(\mathbf{z}),
\end{equation}
we have that, 
\begin{equation}\notag
\widetilde{\Pi}(\alpha \mathbf{z})=\alpha \mathbf{z} \overline{\eta}(\alpha \mathbf{z})+\overline{\alpha \mathbf{z}} \eta(\alpha \mathbf{z})= \alpha \mathbf{z} \overline{\alpha}\,\overline{\eta}(\mathbf{z})+\overline{\alpha}\, \mathbf{\overline{z}} \alpha \eta(\mathbf{z})=\lvert \alpha \rvert ^2 \widetilde{\Pi}(\mathbf{z}). 
\end{equation}
\end{proof}

By the Lemma \ref{le:9} the map $\widetilde{\Pi}$ induces a well-defined projection map: 
\begin{align*}
\Pi: [U_+] & \rightarrow \mathbb{P}_{\mathbb{R}}^n \\
 z & \mapsto \Pi(z)=[\widetilde{\Pi}(\mathbf{z})].
\end{align*}

If we take a real negative vector $\mathbf{z}$ then, since $ \mathbf{\overline{z}}=\mathbf{z}$
\begin{equation} \notag
\eta^{2}=-\langle\mathbf{z},\mathbf{\overline{z}}\rangle=-\langle\mathbf{z},\mathbf{z}\rangle=\lvert \langle\mathbf{z},\mathbf{z}\rangle \rvert,
\end{equation}

then $\overline{\eta}(\mathbf{z})=\eta(\mathbf{z})$, therefore

\begin{equation}\notag
\mathbf{z} \overline{\eta}(\mathbf{z})+ \overline{\mathbf{z}} \eta(\mathbf{z})=2\eta(\mathbf{z}) \mathbf{z}. 
\end{equation}

Thus, we can extend $\Pi$ continuously to all $\Omega =[U_+] \cup \mathbb{H}_{\mathbb{R}}^n$ by defining 

\begin{equation}\notag
\Pi(z)=z, \;\; \mathrm{if}\;\; z\in \mathbb{H}_{\mathbb{R}}^n. 
\end{equation}

The next result follows straightforwardly from Lemma \ref{le:8}. 
\begin{proposition}
The projection $\Pi: \Omega=[U_+] \cup \mathbb{H}_{\mathbb{R}}^n \rightarrow \mathbb{P}_{\mathbb{R}}^n$ is $\mathrm{SO_+}$-equivariant. 
\end{proposition}

\begin{lemma}
The image of the function $\widetilde{\Pi}$ is contained in the space of real negative vectors. Hence the image of the function $\Pi$ is contained in $\mathbb{H}_{\mathbb{R}}^n$.
\end{lemma}

\begin{proof}
We only need to prove that $\widetilde{\Pi}(\mathbf{z})$ is negative for all $\mathbf{z} \in U_+$ because the image of $\widetilde{\Pi}$ is a subset of $\mathbb{R}^{n,1}$ and $\Pi(z)=z$ for all $z \in \mathbb{H}_{\mathbb{R}}^n$.  If $\mathbf{z}$ lies in $U_+$ then 
\begin{align*} 
0<f(\mathbf{z}) &=\langle\mathbf{z},\mathbf{\overline{z}}\rangle \langle\mathbf{\overline{z}},\mathbf{z}\rangle-\langle\mathbf{z},\mathbf{z}\rangle \langle\mathbf{\overline{z}},\mathbf{\overline{z}}\rangle \\
& =-\eta^2(\mathbf{z})\,\overline{\eta}^2(\mathbf{z})-\langle\mathbf{z},\mathbf{z}\rangle^2 \\
& =\lvert \eta(\mathbf{z}) \rvert ^4- \langle\mathbf{z},\mathbf{z}\rangle^2. 
\end{align*}
Therefore  $\lvert \eta(\mathbf{z}) \rvert ^2 > \langle\mathbf{z},\mathbf{z}\rangle$. Using this fact it is not hard to check that 
\begin{align*} 
\langle \widetilde{\Pi}(\mathbf{z}),\widetilde{\Pi}(\mathbf{z})\rangle  
=2\lvert \eta(\mathbf{z}) \rvert ^2(\langle\mathbf{z},\mathbf{z}\rangle-\lvert \eta(\mathbf{z}) \rvert ^2) <0.
\end{align*}
Thus, $\widetilde{\Pi}(\mathbf{z})$ is a negative vector. It follows that $[\widetilde{\Pi}(\mathbf{z})]$ lies in $\mathbb{H}_{\mathbb{R}}^n$ for all $\mathbf{z}\in U_+$ and $\Pi(\Omega)$ is a subset of  $\mathbb{H}_{\mathbb{R}}^n$.
\end{proof}

The following Lemma is the analogous to Corollary 3.4 in \cite{cano2016action} and we have put it here for completeness. 

\begin{lemma}
If $\mathbf{z}$ is a positive vector such that $\mathbf{z}$ and $\mathbf{\overline{z}}$ are linearly independent and $\overleftrightarrow{\mathbf{z},\mathbf{\overline{z}}}$ is parabolic, then $\mathbf{z}\overline{\eta}(\mathbf{z})+\mathbf{\overline{z}}\eta(\mathbf{z}) \in \mathbf{H_z} \cap \mathbf{H}_{\overline{\mathbf{z}}}$.
\end{lemma}

\begin{proof}
We have that $f(\mathbf{z})=0$, so 
\begin{equation}\notag
0=f(\mathbf{z})=\lvert \eta(\mathbf{z}) \rvert^4-\langle\mathbf{z},\mathbf{z}\rangle^2, 
\end{equation}
therefore $\lvert \eta(\mathbf{z}) \rvert^2=\langle\mathbf{z},\mathbf{z}\rangle>0$ because $\mathbf{z}$ is positive. 
Thus 
\begin{eqnarray*}
\langle \mathbf{z},\mathbf{z} \overline{\eta}(\mathbf{z})+\overline{\mathbf{z}}\eta(\mathbf{z}) \rangle & = & \eta(\mathbf{z})\langle \mathbf{z}, \mathbf{z} \rangle+ \overline{\eta}(\mathbf{z}) \langle \mathbf{z}, \overline{\mathbf{z}} \rangle \\
&=& \eta(\mathbf{z})\langle \mathbf{z}, \mathbf{z} \rangle +\overline{\eta}(\mathbf{z})(-\eta^2(\mathbf{z}))\\
&=& \eta(\mathbf{z})\left( \langle \mathbf{z}, \mathbf{z} \rangle- \lvert \eta(\mathbf{z}) \rvert^2 \right) \\
&=& 0 
\end{eqnarray*}
 Analogously, $\langle \overline{\mathbf{z}},\mathbf{z} \overline{\eta}(\mathbf{z})+\overline{\mathbf{z}}\eta(\mathbf{z}) \rangle=0$. Hence, $\mathbf{z}\overline{\eta}(\mathbf{z})+\mathbf{\overline{z}}\eta(\mathbf{z})$ lies in $\mathbf{H_z} \cap \mathbf{H}_{\overline{\mathbf{z}}}$.
\end{proof}

\subsection{The fiber of $\Pi$}
\label{subsec7}

In this subsection we still assume that $n>2$. Consider the function $\Pi$ defined in the Subsection \ref{subsec6}. We compute the preimage under $\Pi$ of the point $o=[0: \cdots :0:1] \in \mathbb{P}_{\mathbb{C}}^n$ and then we find the fiber over any other point using that $\mathrm{SO_+}(n,1)$ acts transitively on $\mathbb{H}_{\mathbb{R}}^n$ and the equivariance of $\Pi$.

\begin{lemma}
Let $z=[z_1: \cdots: z_{n+1}]$ be a point in $\Omega$ such that $\Pi(z)=o$, then the last homogeneous coordinate $z_{n+1}$ is non zero and the quotient $z_j/z_{n+1}$ is purely imaginary for $1 \leq j \leq n$. Moreover, the fiber $L_o=\Pi^{-1}(o)$ is the Lagrangian space 
\begin{equation}\notag
L_o=\{[iy_1:\cdots :iy_n:x_{n+1}] : x_{n+1},y_1, \dots y_n \in \mathbb{R}, x_{n+1} \neq 0\}.
\end{equation}
The boundary $C_o$ of $L_o$ consists of all the points in $\mathbb{P}_{\mathbb{C}}^n$ that can be represented by homogeneous coordinates of the same form but with $x_{n+1}=0$ 
\begin{equation}\notag
C_o=\{[iy_1:\cdots :iy_n:0] : (y_1, \dots ,y_n) \in \mathbb{R}^n \setminus \{(0, \dots,0)\} \}.
\end{equation} 
\end{lemma} 

\begin{proof}
If $z \in \Omega$ and $\Pi(z)=o$ then 
\begin{equation}\notag
\widetilde{\Pi}(\mathbf{z})=\mathbf{z} \overline{\eta}(\mathbf{z})+\mathbf{\overline{z}}\eta(\mathbf{z})=\left(\begin{array}{c}
z_{1}\overline{\eta}+\overline{z_{1}}\eta\\
\vdots\\
z_{n}\overline{\eta}+\overline{z_{n}}\eta\\
z_{n+1}\overline{\eta}+\overline{z_{n+1}}\eta
\end{array}\right)=o=\left(\begin{array}{c}
0\\
\vdots\\
0\\
1
\end{array}\right)
\end{equation}
where $\eta^2=z_{n+1}^2-z_1^2- \cdots -z_n^2$. We obtain the $n+1$ equations
\begin{equation} \label{eq:6}
 0=z_j \overline{\eta}+\overline{z_j}\eta \mathrm{,} \;\; 1 \leq j \leq n,  
\end{equation}
\begin{equation}\label{eq:7}
 1=z_{n+1} \overline{\eta}+\overline{z_{n+1}}\eta.
\end{equation}
We multiply the $n$ equations in \eqref{eq:6} by their corresponding $z_j\eta$ 
\begin{equation} \label{eq:8}
0=z_j \eta(z_j \overline{\eta}+\overline{z_j}\eta)=z_j^2 \lvert \eta \rvert^2+ \lvert z_j \rvert^2 \eta^2 \; \mathrm{for} \; 1 \leq j \leq n.
\end{equation}
 We define $\theta$ such that $\eta^2=\lvert \eta \rvert ^2 e^{i \theta}$, so by Equations \eqref{eq:8} 
\begin{equation}\label{eq:9}
z_j^2=- \lvert z_j \rvert^2 e^{i \theta} \;\; 1 \leq j \leq n.
\end{equation}
We substitute \eqref{eq:9} in $\eta^2$,
\begin{equation} \notag
\lvert \eta \rvert ^2 e^{i\theta}=\eta^2=z_{n+1}^2-{\displaystyle \sum_{j=1}^{n}}z_j^2=z_{n+1}^2-{\displaystyle \sum_{j=1}^{n}} \lvert z_j \rvert ^2 e^{i\theta},
\end{equation}
So, we can write $z_{n+1}^2$ as $\epsilon \lvert z_{n+1} \rvert ^2 e^{i\theta}$ for some $\epsilon \in \{-1,1\}$, and $\lvert \eta \rvert ^2=\epsilon \lvert z_{n+1} \rvert ^2+ {\displaystyle \sum_{j=1}^{n}} \lvert z_j \rvert ^2$, now by these equalities and Equation \eqref{eq:7},  
\begin{align*}
1&=(z_{n+1} \overline{\eta}+\overline{z_{n+1}}\eta)^2 =2(\epsilon +1 )\lvert z_{n+1} \rvert ^2 \lvert \eta \rvert ^2. 
\end{align*}
Therefore, $z_{n+1} \neq 0$ and $\epsilon=1$, then $z_{n+1}^2=\lvert z_{n+1} \rvert^2 e^{i \theta}$ and the quotient $z_j^2/z_{n+1}^2=-\lvert z_{j} \rvert ^2/\lvert z_{n+1} \rvert ^2  $ is negative for $1 \leq j \leq n$. Hence $z_j/z_{n+1}$ is purely imaginary for $1 \leq j \leq n$.
\end{proof}

Now we can compute the general fibre  using the fibre over $o$, in the sake of this we consider $\mathbb{H}_{\mathbb{R}}$ as the unit ball $\{[x_1: \cdots x_n:1]: x_1 ^2+ \cdots + x_n ^2-1 <0\}$ in $\mathbb{P}_{\mathbb{R}}^n$. So, we can describe, see \cite{ratcliffe1994foundations},  the coordinates $(x_1, \cdots ,x_n)$ as the hyperbolic coordinates,   
\begin{align*}
x_1 &=\mathrm{tanh}(t_1)\mathrm{cos}(t_{n})\mathrm{sin}(t_{n-1}) \cdots \mathrm{sin}(t_{2})\\
x_2 &=\mathrm{tanh}(t_1)\mathrm{sin}(t_{n})\mathrm{sin}(t_{n-1}) \cdots \mathrm{sin}(t_{2}) \\
x_3 &=\mathrm{tanh}(t_1)\mathrm{cos}(t_{n-1})\mathrm{sin}(t_{n-2}) \cdots \mathrm{sin}(t_{2}) \\
\vdots \\
x_{n-2} &=\mathrm{tanh}(t_1)\mathrm{cos}(t_4)\mathrm{sin}(t_3)\mathrm{sin}(t_2)
\\
x_{n-1} &=\mathrm{tanh}(t_1)\mathrm{cos}(t_3)\mathrm{sin}(t_2)\\
x_n &=\mathrm{tanh}(t_1)\mathrm{cos}(t_2).  
\end{align*}
 In homogeneous coordinates we obtain 
\begin{align*}
x&=[x_1: \cdots :x_n :1]\\
&=[\mathrm{tanh}(t_1)\mathrm{cos}(t_{n}){\displaystyle \prod_{k=2}^{n-1}}{\mathrm{sin}(t_{k})}: \cdots :\mathrm{tanh}(t_1)\mathrm{cos}(t_2):1],
\end{align*}
as $t_1,...,t_n$ vary, we describe the real hyperbolic space $\mathbb{H}_{\mathbb{R}}^n$ embedded in $\mathbb{H}_{\mathbb{C}}^n \subset \mathbb{P}_{\mathbb{C}}^n$. 

\begin{lemma} \label{le:14}
Consider $x=[\mathrm{tanh}(t_1)\mathrm{cos}(t_{n}){\displaystyle \prod_{k=2}^{n-1}}{\mathrm{sin}(t_{k})}: \cdots :\mathrm{tanh}(t_1)\mathrm{cos}(t_2):1]$ a point in $\mathbb{H}_{\mathbb{R}}^n$. The matrix $A$ described below lies in $\mathrm{SO_+}(n,1)$ and projectively carries $o=[0: \cdots :0:1]$ into $x$.   
\end{lemma} 

We describe the matrix $A$ by displaying its columns. The $j$-th column $A_j$ of $A$ is given as follows:   
\begin{align*}
\small{A_1}=\left(\begin{smallmatrix}
\mathrm{cosh}(t_{1})\mathrm{cos}(t_{n})\prod\limits_{k=2}^{n-1}{\mathrm{sin}(t_{k})}\\
\mathrm{cosh}(t_{1})\mathrm{sin}(t_{n})\prod\limits_{k=2}^{n-1}{\mathrm{sin}(t_{k})}\\
\mathrm{cosh}(t_{1})\mathrm{cos}(t_{n-1})\prod\limits_{k=2}^{n-2}{\mathrm{sin}(t_{k})}\\
\vdots\\
\mathrm{cosh}(t_{1})\mathrm{cos}(t_{3})\mathrm{sin}(t_{2})\\
\mathrm{cosh}(t_{1})\mathrm{cos}(t_{2})\\
\mathrm{sinh}(t_{1}) 
\end{smallmatrix} \right), & \;
  A_{n+1}=\left(\begin{smallmatrix}
\mathrm{sinh}(t_{1})\mathrm{cos}(t_{n}) \prod\limits_{k=2}^{n-1}{\mathrm{sin}(t_{k})}\\
\mathrm{sinh}(t_{1})\mathrm{sin}(t_{n})\prod\limits_{k=2}^{n-1}{\mathrm{sin}(t_{k})}\\
\mathrm{sinh}(t_{1})\mathrm{cos}(t_{n-1})\prod\limits_{k=2}^{n-2}{\mathrm{sin}(t_{k})}\\
\vdots\\
\mathrm{sinh}(t_{1})\mathrm{cos}(t_{3})\mathrm{sin}(t_{2})\\
\mathrm{sinh}(t_{1})\mathrm{cos}(t_{2})\\
\mathrm{cosh}(t_{1})
\end{smallmatrix} \right)
\end{align*}

and

\begin{align*}
A_j=\left(\begin{smallmatrix}
\mathrm{cos}(t_{n}) \mathrm{cos}(t_{j}) \prod\limits_{k=j+1}^{n-1}{\mathrm{sin}(t_{k})}\\
\mathrm{sen}(t_{n}) \mathrm{cos}(t_{j}) \prod\limits_{k=j+1}^{n-1}{\mathrm{sin}(t_{k})}\\
\mathrm{cos}(t_{n-1}) \mathrm{cos}(t_{j}) \prod\limits_{k=j+1}^{n-2}{\mathrm{sin}(t_{k})}\\
\vdots\\
\mathrm{cos}(t_{j+2})\mathrm{sin}(t_{j+1})\mathrm{cos}(t_{j})\\
\mathrm{cos}(t_{j+1})\mathrm{cos}(t_{j})\\
\mathrm{-sin}(t_{j})\\
0\\
\vdots\\
0
\end{smallmatrix}\right) & \; \mathrm{for} \; 2 \leq j \leq n.
\end{align*}

The proof of Lemma \ref{le:14} is straightforward and we omit it. The matrix $A$ allow to translate the fibre over $o$ to the fibre over an arbitrary point in $\mathbb{H}_{\mathbb{R}}^n$ and induces the following basis 
$\mathcal{B}_x$ adapted to $x$: 

\begin{equation}\notag
 \mathcal{B}_x:=\{W_{1}=A_1/\mathrm{cosh}(t_{1}), W_{2}=A_2, \cdots,W_n=A_n, W_{n+1}=A_{n+1}/\mathrm{cosh}(t_{1})\}
\end{equation}

We have scaled  $A_1$ and $A_{n+1}$ by the positive number $1/\mathrm{cosh}(t_{1})$ as in \cite{cano2016action}. The following results follows from applying the matrix $A$ to the fibre over $o$.

\begin{lemma}\notag
If $z$ is a point in $\Omega$ such that $\Pi(z)=x$, where $x \in \mathbb{H}_{\mathbb{R}}^n$ is the point
\begin{equation}\notag
x=[\mathrm{tanh}(t_1)\mathrm{cos}(t_{n})\mathrm{sin}(t_{n-1}) \cdots \mathrm{sin}(t_{2}): \cdots :\mathrm{tanh}(t_1)\mathrm{cos}(t_2):1],
\end{equation} 
then $z$ is the image under $A$ of a point in $L_o$ and $L_x=\Pi^{-1}(x)$ is the Lagrangian space
\begin{equation}\notag
L_x=\{[iy_1W_1+\cdots +iy_nW_n+ix_{n+1}W_{n+1}]:x_{n+1},y_1,...,y_n\in \mathbb{R}, x_{n+1}\neq 0\}.
\end{equation}  
The boundary of $L_x$ is 
\begin{equation} \notag
C_x=\{[iy_1W_1+\cdots +iy_nW_n]:(y_1,...,y_n)\in \mathbb{R}^n \setminus \{0,...,0\} \}.
\end{equation}
\end{lemma}

Now, we prove the Theorem \ref{teo:uno} verifying that the projection $\Pi: \Omega \rightarrow \mathbb{H}_{\mathbb{R}}^{n}$ is an $\mathrm{SO_+}(n,1)$-equivariant smooth fibre bundle.  The result is analogous to Theorem 3.6 in \cite{cano2016action}.
\begin{proof} [Proof of Theorem \ref{teo:uno}]
Consider a point $x \in \mathbb{H}_{\mathbb{R}}^{n}$ distinct of the origin. We can find an open ball $U(x)$ centered at $x$ that does not contain the origin. So, we define $\varphi:\Pi^{-1}(U(x)) \rightarrow U(x) \times L_o$ given by $z=A_yl_o \mapsto (y,l_o)$. It is not hard to see that $\varphi$ is smooth and has smooth inverse. Also, since $\Pi$ is equivariant $\Pi(z)=\Pi(A_yl_o)=A_yo=y$, $\Pi$ agrees with the projection on the first factor of $\varphi$. For the origin, we take an open ball $U_x$ centered at a point $x \neq o$ that does not contain $o$. Then $A_x^{-1}U_x$ is an open neighborhood of $o$. We take the function $\varphi_o:\Pi^{-1}(A_x^{-1}U_x) \rightarrow A_x^{-1}U_x \times L_o$ given by $z=A_{A_x^{-1}y}l_o \mapsto (A_x^{-1}y,l_o)$, again it is not hard to see that $\varphi_o$ is smooth and has smooth inverse, also $\Pi$ agrees with the projection on the first factor of $\varphi_o$ because $\Pi(z)=\Pi(A_{A_x^{-1}y}l_o)=y$. Thus, $\Pi$ is a smooth fibre bundle with connected fibre $L_o$.  
\end{proof}

\section{A $\mathbb{R}$-circle in $\mathbb{P}_{\mathbb{C}}^{n}$}
\label{sec4}

\subsection{The projection}
\label{subsec8}
For this subsection we assume $n>2$. Consider the $\mathbb{R}$-circle $\partial\mathcal{H}_\mathbb{R}^2$ defined as the set
\begin{equation}\notag
\{[w_1:w_2:0:\cdots:0:w_{n+1}]\in \mathbb{P}_{\mathbb{C}}^{n}:w_1,w_2,w_{n+1}\in \mathbb{R}, w_1^2+w_2^2-w_{n+1}^2 =0\}.
\end{equation}

The circle $\partial\mathcal{H}_\mathbb{R}^2$ is the set of null vectors of the subspace 
\begin{equation}\notag
\mathcal{P}_\mathbb{R}^2:=\{[w_1:w_2:0:\cdots:0:w_{n+1}] \in \mathbb{P}_{\mathbb{C}}^{n}:w_1,w_2,w_{n+1}\in \mathbb{R}\}
\end{equation}
 which is a copy of $\mathbb{P}_\mathbb{R}^2$. Moreover, the circle $\partial\mathcal{H}_\mathbb{R}^2$ is the boundary of the disk 
\begin{equation}\notag
\mathcal{H}_\mathbb{R}^2:=\{w \in \mathcal{P}_\mathbb{R}^2 : w_1^2+w_2^2-w_{n+1}^2 <0\}.
\end{equation}
 We are interested again in the set obtained as union of hyperplanes 
\begin{equation}\notag
\Lambda=\bigcup_{p\in \partial\mathcal{H}_\mathbb{R}^2}H_p. 
\end{equation} 

Let $\mathbf{z}=(z_1 \cdots, z_{n+1})$ be a point in $\mathbb{C}^{n,1}$, we study the intersection of $H_{z}$ with $\mathcal{P}_\mathbb{R}^2$, where $z=[\mathbf{z}]$. If $\mathbf{z}$ has the form $(0,0,z_3,...,z_n,0)$ then $H_{z} \cap \mathcal{P}_{\mathbb{R}}^2=\mathcal{P}_{\mathbb{R}}^2$, so $z$ lies in $\Lambda$. Conversely, if $H_{z} \cap \mathcal{P}_{\mathbb{R}}^2=\mathcal{P}_{\mathbb{R}}^2$ then, it is not hard to see that, $\mathbf{z}$ must have the form $(0,0,z_3,...,z_n,0)$. In order to study the case when $\mathbf{z}$ has not the form $(0,0,z_3,...,z_n,0)$, we define the projection 
\begin{gather*} 
\widetilde{Q} : \; \mathbb{C}^{n,1} \rightarrow \mathbb{C}^{2,1}  \\
\mathbf{z}= (z_1,z_2,...,z_{n+1}) \mapsto \widetilde{Q}(\mathbf{z})=(z_1,z_2,z_{n+1}).
\end{gather*}
The kernel of the projection $\widetilde{Q}$ is the set 
\begin{equation}\notag
\widetilde{\Lambda}_0=\{(z_1,z_2,...,z_n,z_{n+1}) \in \mathbb{C}^{n,1}:z_1=z_2=z_{n+1}=0\}\cong \mathbb{C}^{n-2}. 
\end{equation}
So, the map $\widetilde{Q}$ induces a projection 
\begin{gather*}
Q: \; \mathbb{P}_\mathbb{C}^{n} \setminus [\widetilde{\Lambda}_0] \rightarrow \mathbb{P}_\mathbb{C}^{2} \\
z=[z_1:z_2:...:z_{n+1}] \mapsto Q(z)=[z_1:z_2:z_{n+1}].
\end{gather*} 
The maps $\widetilde{Q}$ and $Q$ allow to restrict this study to the case treated in \cite{cano2016action} and \cite{barrera2019chains}. We denote the set $[\widetilde{\Lambda}_0]$ by $\Lambda_0$. If $\mathbf{z}$ has not the form $(0,0,z_3,...,z_n,0)$ then we have two possibilities  
\begin{enumerate}
	\item $\widetilde{Q}(\mathbf{z})$ and $\widetilde{Q}(\mathbf{\overline{z}})$ are linearly dependent and $H_{z} \cap \mathcal{P}_{\mathbb{R}}^2$ is a projective line in $\mathcal{P}_{\mathbb{R}}^2$, or 
	\item $\widetilde{Q}(\mathbf{z})$ and $\widetilde{Q}(\mathbf{\overline{z}})$ are linearly independent and $H_{z} \cap \mathcal{P}_{\mathbb{R}}^2$ is a point in $\mathcal{P}_{\mathbb{R}}^2$.
\end{enumerate}
 
In the first case when the intersection is a projective line $l:=H_{z} \cap \mathcal{P}_{\mathbb{R}}^2$ we have  two possibilities:  
  \begin{itemize}
	\item[1.1] the projective line $l$ is a subset of $\mathcal{P}_{\mathbb{R}}^2 \setminus \overline{\mathcal{H}_{\mathbb{R}}^2}$, and $z$ lies in $\Omega$,
	\item[1.2] the projective line $l$ intersects $\overline{\mathcal{H}_{\mathbb{R}}^2}$, and  $z$ lies in $\Lambda$. 
		\end{itemize}
		
 In the second case when the intersection is a point, we have the following possibilities:   		
\begin{itemize}
	\item[2.1] the point is contained in $\mathcal{P}_{\mathbb{R}}^2 \setminus \overline{\mathcal{H}_{\mathbb{R}}^2}$, and $z$ lies in $\Omega$,
	\item[2.2] the point is contained in $\mathcal{H}_{\mathbb{R}}^2$, and $z$ lies in $\Omega$, 
	\item[2.3] the point is contained in $\partial\mathcal{H}_\mathbb{R}^2$, and $z$ lies in $\Lambda$.
		\end{itemize}
		We denote the union of the tangent lines to $\partial \mathbb{H}_{\mathbb{C}}^2$ at points of $\partial \mathbb{H}_{\mathbb{R}}^2$ in $\mathbb{P}_{\mathbb{C}}^2$ by $\Lambda_{(2)}$. The Corollary 2.5 and the Theorem 2.7 in \cite{cano2016action} give descriptions of the sets $\Lambda_{(2)}$ and $\Omega_{(2)}:=\mathbb{P}_{\mathbb{C}}^2 \setminus \Lambda_{(2)}$, respectively in terms of the function $f$ defined in the Subsection \ref{subsec3}. The descriptions are similar to those given in the proof of Theorem 2 and in Proposition \ref{te:de}.  
\begin{lemma} \label{le:lam}
The set $\Lambda$ is equal to $Q^{-1}(\Lambda_{(2)}) \cup \Lambda_0$. 
\end{lemma}
\begin{proof}
 Note that $\mathbb{P}_{\mathbb{C}}^n=\Lambda_0 \cup Q^{-1}(\Lambda_{(2)}) \cup Q^{-1}(\Omega_{(2)})$. Take $z=[z_1:z_2:\cdots:z_{n+1}] \in \Lambda$. Suppose that $z \in Q^{-1}(\Omega_{(2)})$, since $z \in \Lambda$ then there is a $w=[w_1:w_2:0:\cdots:0:w_{n+1}] \in \partial\mathcal{H}_\mathbb{R}^2$ such that $z$ belongs to $H_w$, the tangent hyperplane to $\partial\mathbb{H}_{\mathbb{C}}^n$ at the point $w$. So, 
\begin{equation}\notag
z_1\overline{w_1}+z_2\overline{w_2}-z_{n+1}\overline{w_{n+1}}=0
\end{equation}
this implies that $Q(z)=[z_1:z_2:z_{n+1}]$ must be in the projective line tangent to $\partial \mathbb{H}_{\mathbb{R}}^2$ at the point $Q(w)=[w_1:w_2:w_{n+1}]$, but this is impossible since $Q(z) \in \Omega_{(2)}$. Thus $z \in Q^{-1}(\Lambda_{(2)}) \cup \Lambda_0$. On the other hand, we know that $\Lambda_0 \subset \Lambda$. Now, we take $z \in Q^{-1}(\Lambda_{(2)})$, then $Q(z)$ must be in the projective line tangent to $\partial \mathbb{H}_{\mathbb{R}}^2$ at some point $w^\prime=[w^\prime_1:w^\prime_2:w^\prime_{n+1}]$, so $z$ belongs to $H_w$, where $w=[w^\prime_1:w^\prime_2:0:\cdots:0:w^\prime_{n+1}]$ lies in $Q^{-1}(w^\prime) \cap \partial\mathcal{H}_\mathbb{R}^2$. That is  $z$ belongs to $\Lambda$.    
\end{proof}
  
By the Lemma \ref{le:lam} the set $\Omega:= \mathbb{P}_{\mathbb{C}}^n \setminus \Lambda$ is equal to the set $Q^{-1}(\Omega_{(2)})$, where $\Omega_{(2)}= \mathbb{P}_{\mathbb{C}}^2 \setminus \Lambda_{(2)}$. Also Cano et al. \cite{cano2016action} show that $\Omega_{(2)}$ has three components

\begin{align*}
\Omega_{(2)}^{0}\!&\!=\!\{p\!\in\!\mathbb{P}_{\mathbb{C}}^2\!:\! H_p\!\cap \!\mathbb{P}_{\mathbb{R}}^2\!\subset\!\mathbb{P}_{\mathbb{R}}^2\!\setminus\! \overline{\mathbb{H}_{\mathbb{R}}^2}\}, \\
\Omega_{(2)}^{1-}\!&\!=\!\{p\!=\![p_1\!:\! p_2\!:\! p_3]\!\in\!\mathbb{P}_{\mathbb{C}}^2\!:\! H_p\!\cap\!\mathbb{P}_{\mathbb{R}}^2\; \mathrm{is\;a\;point\;in\;} \mathbb{H}_{\mathbb{R}}^2\; \mathrm{and}\; \Bigg\lvert \begin{matrix} \mathrm{Re}(p_1)&\!\mathrm{Re}(p_2)\\ \mathrm{Im}(p_1)&\!\mathrm{Im}(p_2) \end{matrix} \Bigg\rvert\!<\!0 \},\\
\Omega_{(2)}^{1+}\!&\!=\!\{p\!=\![p_1\!:\! p_2\!:\! p_3]\!\in\!\mathbb{P}_{\mathbb{C}}^2\!:\! H_p\!\cap\!\mathbb{P}_{\mathbb{R}}^2\; \mathrm{is\;a\;point\;in\;} \mathbb{H}_{\mathbb{R}}^2\; \mathrm{and}\; \Bigg\lvert \begin{matrix} \mathrm{Re}(p_1)&\!\mathrm{Re}(p_2)\\ \mathrm{Im}(p_1)&\!\mathrm{Im}(p_2) \end{matrix} \Bigg\rvert\!>\!0 \}.
\end{align*} 

each being diffeomorphic to an open $4$-ball. The sets $Q^{-1}(\Omega_{(2)}^{0})$, $Q^{-1}(\Omega_{(2)}^{1-})$ and $Q^{-1}(\Omega_{(2)}^{1+})$ are non-empty, open and disjoint.
 
\begin{proposition} \label{pro:7}
Let $n>2$ be an integer. The set $\mathbb{P}_{\mathbb{C}}^n \setminus \Lambda_0$ is diffeomorphic to $\mathbb{P}_{\mathbb{C}}^2 \times \mathbb{C}^{n-2}$. 
\end{proposition}
\begin{proof}
Observe that  
\begin{equation}\notag
\mathbb{P}_{\mathbb{C}}^n\!\setminus\!\{[1\!:\! z_2\!:\!\cdots\!:\! z_{n+1}]: z_2,\!\cdots\!,\!z_{n+1}\!\in\!\mathbb{C}\}\!=\!\{[0\!:\! z_2\!:\!\cdots\!:\! z_{n+1}]: z_2,\!\cdots\!,\! z_{n+1}\!\in\!\mathbb{C}\}.
\end{equation}
and analogously
\begin{align*}
\!&\{[0\!:\! z_2\!:\!\cdots\!:\! z_{n+1}]: z_2,\!\cdots\! z_{n+1}\!\in\!\mathbb{C}\}\setminus \!\{[0\!:\! 1\!:\! z_3\!:\!\cdots\!:\! z_{n+1}]: z_3,\!\cdots\!,\! z_{n+1}\!\in\!\mathbb{C}\}=\\
&\{[0\!:\! 0\!:\! z_3\!:\!\cdots\!:\! z_{n+1}]: z_3,\!\cdots\!,\! z_{n+1}\!\in\!\mathbb{C}\}
\end{align*} 
and
\begin{align*}
\!&\{[0\!:\! 0\!:\! z_3\!:\!\cdots\!:\! z_{n+1}]: z_3,\!\cdots\!,\! z_{n+1}\!\in\!\mathbb{C}\}\setminus \!\{[0\!:\! 0\!:\! z_3\!:\!\cdots\!:\! z_{n}\!:\!1]: z_3,\!\cdots\!,\! z_{n}\!\in\!\mathbb{C}\}=\\
&\{[0\!:\! 0\!:\! z_3\!:\!\cdots\!:\! z_{n}\!:\!0]: z_3,\!\cdots\!,\! z_{n}\!\in\!\mathbb{C}\}=\Lambda_0.
\end{align*}
Thus, 
\begin{align*}
\mathbb{P}_{\mathbb{C}}^n\setminus \Lambda_0=& \\
 &\{[1\!:\! z_2\!:\!\cdots\!:\! z_{n+1}]: z_2,\!\cdots\!,\!z_{n+1}\!\in\!\mathbb{C}\} \cup \\
 &\{[0\!:\! 1\!:\! z_3\!:\!\cdots\!:\! z_{n+1}]: z_3,\!\cdots\!,\! z_{n+1}\!\in\!\mathbb{C}\} \cup \\
 &\{[0\!:\! 0\!:\! z_3\!:\!\cdots\!:\! z_{n}\!:\!1]: z_3,\!\cdots\!,\! z_{n}\!\in\!\mathbb{C}\}. 
\end{align*}
so,
\begin{align*}
\mathbb{P}_{\mathbb{C}}^n\setminus \Lambda_0 \cong& \\
 &(\{[1\!:\! z_2\!:\!0\!:\!\cdots\!:\!0\!:\! z_{n+1}]: z_2,\!z_{n+1}\!\in\!\mathbb{C}\} \cup \\
 &(\{[0\!:\! 1\!:\!0\!:\!\cdots\!:\!0\!:\! z_{n+1}]:z_{n+1}\!\in\!\mathbb{C}\} \cup \\
 &\{[0\!\cdots\!:\! 0\!:\!1]\!\}) \times \mathbb{C}^{n-2}\\ 
\cong & \; \mathbb{P}_{\mathbb{C}}^2 \times \mathbb{C}^{n-2}.
\end{align*}
\end{proof}
\begin{lemma} \label{le:17}
The four-tuples $(Q^{-1}(\Omega_{(2)}^{0}),\Omega_{(2)}^{0},Q,\mathbb{C}^{n-2})$, $(Q^{-1}(\Omega_{(2)}^{1-}),\Omega_{(2)}^{1-},Q,\mathbb{C}^{n-2})$ and $(Q^{-1}(\Omega_{(2)}^{1+}),\Omega_{(2)}^{1+},Q,\mathbb{C}^{n-2})$ are trivial fibre bundles. 
\end{lemma}

\begin{proof}
By reasoning similar to that given in the proof of the Proposition \ref{pro:7} the set $Q^{-1}(\Omega_{(2)}^{0})$ is diffeomorphic to $\Omega_{(2)}^{0} \times \mathbb{C}^{n-2}$, and the following diagram commutes \\   

\centerline{
\xymatrix{ 
Q^{-1}(\Omega_{(2)}^{0}) \ar[d]_{Q}  \ar[r] & \Omega_{(2)}^{0} \times \mathbb{C}^{n-2} \ar[dl]^{\mathrm{Proj}}\\
\Omega_{(2)}^{0}
} }
The other cases are treated similarly. 
\end{proof}

A consequence of the previous lemma is the connectedness of the sets $Q^{-1}(\Omega_{(2)}^{0})$, $Q^{-1}(\Omega_{(2)}^{1-})$ and $Q^{-1}(\Omega_{(2)}^{1+})$. Thus we have shown that $\Omega$ has three connected components each component diffeomorphic to an open $n$-dimensional complex ball. 

\section{The other cases}
\label{sec5}

\subsection{A real $(m-1)$-sphere in $\mathbb{P}_{\mathbb{C}}^n$}
\label{sec6}

The work \cite{cano2016action} of Cano et al. deals with the case m=n=2. The Section \ref{sec3} cover the case $m=n$ for $n>2$ and the Section \ref{sec4} the case $m=2$ for $n>2$. In this section we study the remaining cases.\\ 

We use the same notation that in the Section \ref{sec4}, that is for $2<m<n$, we write    
\begin{align*}
\mathcal{P}_\mathbb{R}^m &:=\left\{[w_1:\cdots:w_m:0:\cdots:0:w_{n+1}]\in \mathbb{P}_{\mathbb{C}}^{n}:w_k\in\mathbb{R} \mathrm{\;for\;all\;}k\right\} \\
\partial\mathcal{H}_\mathbb{R}^m &:=\left\{[w_1:\cdots:w_m:0:\cdots:0:w_{n+1}]\in \mathcal{P}_\mathbb{R}^m:\; {\displaystyle \sum_{k=1}^{m}}w_k^2-w_{n+1}^2 =0\right\} \\
\mathcal{H}_\mathbb{R}^m &:=\left\{[w_1:\cdots:w_m:0:\cdots:0:w_{n+1}]\in \mathcal{P}_\mathbb{R}^m:\; {\displaystyle \sum_{k=1}^{m}}w_k^2 -w_{n+1}^2 <0 \right\} 
\end{align*}

for the real $m$-projective space, the real $(m-1)$-sphere and the real $m$-ball in $\mathbb{P}_\mathbb{C}^n$, respectively. Again we consider the set obtained as union of hyperplanes 

\begin{equation}\notag
\Lambda=\bigcup_{p\in \partial\mathcal{H}_\mathbb{R}^m}H_p.
\end{equation}

For this purpose we examine the intersection of $H_{z}$ with $\mathcal{P}_\mathbb{R}^m$, where $\mathbf{z}$ is a point in $\mathbb{C}^{n,1}$ and $z=[\mathbf{z}]$.  The next result follows from the same arguments explained in the third paragraph of Subsection \ref{subsec8}.
 
\begin{lemma}
The intersection $H_{z} \cap \mathcal{P}_\mathbb{R}^m$ is equal to $\mathcal{P}_\mathbb{R}^m$ if and only if $\mathbf{z}$ has the form $(0,\cdots,0,z_{m+1},\cdots,z_n,0)$.
\end{lemma}

We now suppose that $\mathbf{z}$ is not in the following subspace 
\begin{equation}\notag
\widetilde{\Lambda}_0:=\{\mathbf{w}=(w_1,\cdots,w_{n+1})\in \mathbb{C}^{n+1}: w_1\!=\!\cdots\!=\! w_m\!=\! w_{n+1}\!=\!0\}\cong \mathbb{C}^{n-m+1},
\end{equation} 

it is useful to consider the projection 
\begin{gather*}
\widetilde{Q}_m: \;\mathbb{C}^{n,1} \rightarrow \mathbb{C}^{m,1} \\
\mathbf{z}=(z_1,z_2,...,z_{n+1}) \mapsto \widetilde{Q}_m(\mathbf{z})=(z_1, \cdots,z_m,z_{n+1}).
\end{gather*}

 The kernel of the projection $\widetilde{Q}_m$ is the set $\widetilde{\Lambda}_0$, so the map $\widetilde{Q}_m$ induces the projection 
\begin{gather*}
Q_m : \;\mathbb{P}_\mathbb{C}^{n} \setminus [\widetilde{\Lambda}_0] \rightarrow \mathbb{P}_\mathbb{C}^{m} \\
z=[z_1,z_2,...,z_{n+1}] \mapsto Q_m(z)=[z_1,\cdots,z_m,z_{n+1}].
\end{gather*} 
Again, we write $\Lambda_0$ for $[\widetilde{\Lambda}_0]$. We can distinguish two cases 
\begin{enumerate}
	\item $\widetilde{Q}_m(\mathbf{z})$ and $\widetilde{Q}_m(\mathbf{\overline{z}})$ are linearly dependent (or equivalently $Q_m(z)$ lies in $\mathcal{P}_\mathbb{R}^m$) and $H_{z} \cap \mathcal{P}_{\mathbb{R}}^m$ is a real projective hyperplane in $\mathcal{P}_{\mathbb{R}}^m$, or 
	\item $\widetilde{Q}_m(\mathbf{z})$ and $\widetilde{Q}_m(\mathbf{\overline{z}})$ are linearly independent and $H_{z} \cap \mathcal{P}_{\mathbb{R}}^m$ is a real projective subspace of dimension $m-2$ in $\mathcal{P}_{\mathbb{R}}^m$.
\end{enumerate}

We have already dealt with the case of a real $(m-1)$-sphere in the $m$-dimensional projective space, so we can use the projection $Q_m$ as in the  Section \ref{sec4} to describe the sets $\Lambda$ and $\Omega=\mathbb{P}_{\mathbb{C}}^n \setminus \Lambda$.  We denote by $\Lambda_{(m)}$ the union of all complex hyperplanes in $\mathbb{P}_{\mathbb{C}}^m$ tangent to $\partial\mathbb{H}_{\mathbb{C}}^m$ at points in $\partial\mathbb{H}_{\mathbb{R}}^m$ and by $\Omega_{(m)}$ to $\mathbb{P}_{\mathbb{C}}^m \setminus \Lambda_{(m)}$.     

\begin{lemma} \label{le:19}
The set $\Lambda$ is equal to $Q_m^{-1}(\Lambda_{(m)}) \cup \Lambda_0$. Also, the set $\Omega$ is equal to $Q_m^{-1}(\Omega_{(m)})$ and it is connected. 
\end{lemma}
  
\begin{proof}
Arguing in the same manner as in the proof of Lemma \ref{le:lam} we have the equality $\Lambda=Q_m^{-1}(\Lambda_{(m)}) \cup \Lambda_0$, and therefore also the equality $\Omega=Q_m^{-1}(\Omega_{(m)})$. By using arguments similar to those given in the proof of Lemma \ref{le:17} it is not hard to see that the four-tuple $(Q^{-1}(\Omega_{(m)}),\Omega_{(m)},Q,\mathbb{C}^{n-m})$ is a trivial fibre bundle. We know that $\Omega_{(m)}$ is connected, so we conclude that $\Omega$ is connected.
\end{proof}

\section*{\textbf{Declarations}}
\noindent
\textbf{Acknowledgements}
E. Montiel would like to thank to UADY's Facultad de Matem\'aticas for the kindness and the facilities provided during his stay.

\textbf{Funding.} 
The research of W. Barrera and J. P. Navarrete has been supported by the CONACYT, ``Proyecto Ciencia de Frontera'' 2019--21100 via Faculty of Mathematics, UADY, M\'exico. The research of E. Montiel has been supported by the CONACYT.
\noindent
\\
\textbf{Conflict of interests} The authors declare that they have no conflict of interest.

\noindent
\textbf{Authors' contributions}
All authors have contributed equally to the paper.

\bibliographystyle{abbrv}

\end{document}